\documentclass[11pt]{amsart}
  \usepackage{graphicx,amsfonts,layout}
\usepackage{amssymb}
\usepackage{amscd}
\usepackage{amsmath}
\usepackage[all]{xy}
\usepackage{mathrsfs}
\usepackage{graphicx}
 \usepackage{color}

\theoremstyle{plain}
\newtheorem{theorem}{Theorem}[section]
\newtheorem{lemma}[theorem]{Lemma}
\newtheorem{corollary}[theorem]{Corollary}
\newtheorem{proposition}[theorem]{Proposition}

\newtheorem{conjecture}[theorem]{Conjecture}

\newtheorem{example}[theorem]{Example}

\newtheorem{remark}[theorem]{Remark}

  1

\DeclareMathOperator{\Gal}{Gal}

\DeclareMathOperator{\Hom}{Hom}

\DeclareMathOperator{\im}{im}

\newcommand{\CC}{\mathbb{C}}

\newcommand{\GG}{\mathbb{G}}

\newcommand{\QQ}{\mathbb{Q}}
\newcommand{\Q}{\mathbb{Q}}
\newcommand{\RR}{\mathbb{R}}

\newcommand{\ZZ}{\mathbb{Z}}

\newcommand{\calF}{\mathcal{F}}

\newcommand{\G}{\mathcal{G}}

\def\bigcapp{\raise1ex\hbox{\rotatebox{180}{$\biguplus$}}}
 \def\bigcappd{\raise1ex\hbox{\rotatebox{180}{$\displaystyle\biguplus$}}}

\begin{document}

\title[]{On circular distributions and  \\ a conjecture of Coleman}

\author{David Burns and Soogil Seo}

\begin{abstract} We investigate a conjecture of Robert Coleman concerning the module of circular distributions.
\end{abstract}

\address{King's College London,
Department of Mathematics,
London WC2R 2LS,
U.K.}
\email{david.burns@kcl.ac.uk}

\address{Department of Mathematics,
Yonsei University,
Seoul,
Korea.}
\email{soogil.seo@yonsei.ac.kr}



\maketitle

%

\section{Introduction}\label{intro} In the Introduction to \cite{kl} Kubert and Lang describe how the general theory of distributions has a prominent role in number theory research and is strongly influenced by the classical theory of cyclotomic numbers in abelian fields. A similar point is also made by Washington in \cite[Chap. 12]{wash}.  

This article is concerned with an important class of such distributions. To give some details we fix an algebraic closure $\QQ^c$ of $\QQ$ and for each natural number $m$ write $\mu_m$ for the group of $m$-th roots of unity in $\QQ^c$ and set $\mu_m^* := \mu_m\setminus \{1\}$. We also write $\mu_\infty$ for the union of $\mu_m$ over all $m$, set $\mu_\infty^* := \mu_\infty\setminus \{1\}$ and write $\mathcal{F}$ for the multiplicative group comprising functions from $\mu_\infty^*$ to $\QQ^{c,\times}$ that respect the natural action of $\Gal(\QQ^c/\QQ)$.

Then, in the 1980s Coleman defined a `circular distribution', or `distribution' for short in the sequel, to be a function $f$ in $\calF$  that satisfies the relation
\begin{equation}\label{defining prop} \prod_{\zeta^a = \varepsilon} f(\zeta) = f(\varepsilon)\end{equation}
for all natural numbers $a$ and all $\varepsilon$ in $\mu_\infty^*$. (A similar notion was also subsequently introduced by Coates in \cite{coates} in the context of abelian extensions of imaginary quadratic fields.)

We write $\calF^{\rm d}$ for the subgroup of $\calF$ comprising all distributions and further define the group of `strict distributions' to be the subgroup $\mathcal{F}^{\rm sd}$ of $\mathcal{F}^{\rm d}$ comprising distributions that satisfy the congruence relation
\begin{equation}\label{defining prop2} f(\varepsilon\cdot \zeta) \equiv f(\zeta) \,\,\text{ modulo all primes above $\ell$}\end{equation}
for all natural numbers $n$, all primes $\ell$ that are coprime to $n$, all $\varepsilon$ in $\mu_\ell$ and all $\zeta$ in $\mu_n^*$.

It is clear that each of the groups $\mathcal{F}, \mathcal{F}^{\rm d}$ and $\mathcal{F}^{\rm sd}$ is naturally a module over the ring
\[ R := \varprojlim_n \ZZ[\Gal(\QQ(n)/\QQ)],\]
where we write $\QQ(n)$ for the field $\QQ(\mu_n)$ and the transition morphisms in the inverse limit are induced by the natural restriction maps.

As far as explicit examples are concerned, the function $\Phi$ in $\mathcal{F}$ that is defined by setting
\[ \Phi(\zeta) := 1-\zeta\, \, \, \text{ for all $\zeta$ in $\mu_\infty^*$}\]
belongs to $\mathcal{F}^{\rm sd}$. We shall write $\mathcal{F}^{\rm c}$ for the $R$-submodule of $\mathcal{F}^{\rm sd}$ generated by $\Phi$ and refer to elements of $\mathcal{F}^{\rm c}$ as `cyclotomic distributions'.

In addition, Coleman observed that for any non-empty set $\Pi$ of odd prime numbers the function $\delta_\Pi$ in $\mathcal{F}$ with
\[ \delta_\Pi(\zeta) := \begin{cases} -1, &\text{if the order of $\zeta$ is divisible only by primes that belong to $\Pi$,}\\
                                     1, &\text{otherwise}\end{cases}\]
belongs to $\mathcal{F}^{\rm d}$ and, in addition, that the only such function belonging to $\mathcal{F}^{\rm sd}$ is $\delta_{\rm odd} :=
\delta_{\Pi_{\rm odd}}$ where $\Pi_{\rm odd}$ denotes the set of all odd primes. We shall write $\mathcal{D}$ for the $R$-submodule of $\mathcal{F}^{\rm d}$ generated by the set of all functions of the form $\delta_\Pi$ and refer to these functions as `Coleman distributions'.

Then, in 1989, Coleman was led to make the following remarkable conjecture.


\begin{conjecture}[{Coleman}]\label{coleman conj} $\mathcal{F}^{\rm d} = \mathcal{D} + \mathcal{F}^{\rm c}.$ \end{conjecture}

This conjecture was  motivated by the archimedean characterization of cyclotomic units that Coleman had earlier given in \cite{coleman2} and hence to attempts to understand a globalized version of the fact that all norm-compatible families of units in towers of local cyclotomic fields arise by evaluating a `Coleman power series' at roots of unity, as had been proved in \cite{coleman}.

The second author was subsequently able to show in \cite{Seo3} and \cite{Seo4} both that $\mathcal{D}$ is equal to the full torsion subgroup of $\mathcal{F}^{\rm d}$, and hence that the torsion subgroup of $\mathcal{F}^{\rm sd}$ is generated by the single Coleman distribution $\delta_{\rm odd}$, and also that $\mathcal{F}^{\rm c}$ is torsion-free.

This meant, in particular, that Coleman's Conjecture was valid if and only if the inclusion $\mathcal{F}^{\rm c} \subseteq \mathcal{F}^{\rm d}$ induces an identification
\begin{equation}\label{corrected coleman} \mathcal{F}^{\rm c} = \mathcal{F}^{\rm d}_{\rm tf} \end{equation}
where we write $\mathcal{F}_{\rm tf}^{\rm d}$ for the quotient of $\mathcal{F}^{\rm d}$ by its torsion subgroup $\mathcal{D}$.


However, this problem has been found to be difficult and all existing results in this direction (see, for example, work of the second author in \cite{Seo1, Seo2, Seo3} and the related results of Saikia in \cite{saikia}) have in effect only considered the much weaker question of whether, for any given distribution $f$ and any given $\zeta$ in $\mu_\infty^*$ there exists an element $r_{f,\zeta}$ of $R$ such that $f(\zeta) = \Phi(\zeta)^{r_{f,\zeta}}$?

By contrast, we shall now develop techniques that allow a systematic investigation of the conjectural equality (\ref{corrected coleman}) itself.  

To give an example of the sort of results that we are able to prove we write $\widehat{R}$ for the profinite completion of $R$. We also denote complex conjugation by $\tau$ and regard it as an element of $R$ in the obvious way.

We write $\widehat{\ZZ}(1)$ for the inverse limit of the groups $\mu_m$ with respect to the transition morphisms $\mu_m\to \mu_{m'}$ for each divisor $m'$ of $m$ that are given by raising to the power $m/m'$. We observe that $\widehat{\ZZ}(1)$ is naturally an $R$-module.

Then the following result will be proved in \S\ref{proof section}. 

\begin{theorem}\label{main result} There exists a canonical commutative diagram of $R$-modules of the form
\begin{equation*}
\xymatrix{ \widehat{\ZZ}(1) \ar@{^{(}->}[r] & \calF_{\rm tf}^{\rm d} \ar^{\!\!\!\!\!\!\!\kappa}[r] &\widehat R (1+\tau)\\
 \widehat{\ZZ}(1)\ar@{=}[u] \ar@{^{(}->}[r] & \calF^{\rm c} \ar@{->>}[r] \ar@{^{(}->}[u] & R(1+\tau). \ar@{^{(}->}[u]}
\end{equation*}
\end{theorem}
%

%

\noindent{}(The two vertical arrows in this diagram are the natural inclusions and all other maps will be defined explicitly in the course of the proof.) 

This result is of interest for several reasons.

Firstly, for example, it reduces the study of $\calF_{\rm tf}^{\rm d}$ to the study of $(1+\tau)\calF_{\rm tf}^{\rm d}$ and identifies this module with a submodule of the profinite completion of $(1+\tau)\calF^{\rm c}$, thereby showing that Coleman's Conjecture is, in a natural sense, true `everywhere locally'.

Secondly, Theorem \ref{main result} will allow us to establish (in Theorem \ref{new criterion}) an explicit `$p$-adic' criterion for a distribution to be cyclotomic that is very reminiscent of the `archimedean boundedness' criterion used in \cite{coleman2} to characterise cyclotomic units (and which itself motivated Coleman's study of circular distributions).

Thirdly, our approach gives an explicit description of the image of the map $\kappa$ in Theorem \ref{main result}, and hence of the quotient group $Q:= \mathcal{F}^{{\rm d}}_{\rm tf}/\mathcal{F}^{{\rm c}}$, and thereby leads (in Theorem \ref{image result} and Proposition \ref{image result 2}) to an explicit condition concerning the Galois structure of the Selmer groups of $\mathbb{G}_m$ that is equivalent to the validity of Coleman's conjectural equality (\ref{corrected coleman}).

At the same time, it also allows us to show for every prime $p$ that $Q$ identifies with a submodule of a canonical uniquely $p$-divisible group and, in addition, to give an explicit criterion for the image in $Q$ of a distribution to be $p$-divisible.  

Further, it seems to us possible  that further development of the methods introduced here may allow one to prove that $Q$ is itself uniquely divisible, and for all practical purposes concerning applications to the theory of Euler systems this would be enough (see the discussion in \S\ref{final sec}).

In the course of proving Theorem \ref{main result} we shall also establish several auxiliary results that are perhaps themselves of interest.

Such results include an unconditional proof of the main result of the second author in \cite{Seo1} that in loc. cit. is only proved modulo the validity of Greenberg's conjecture on the ideal class groups of real abelian fields (see Theorem \ref{first step} and Remark \ref{soogil rem}), 
 a complete characterization of torsion-valued distributions (see Theorem \ref{reduction result}) and a more conceptual proof of the fact that $\mathcal{D}$ is equal to the torsion subgroup of $\mathcal{F}^{\rm d}$, as discussed above (see Remark \ref{more concept}).


Finally, we note that the short exact sequence in the lower row of the diagram in Theorem \ref{main result} splits after inverting $2$, but not otherwise. 

In fact, an investigation of the submodule $\mathcal{F}^{{\rm d},\tau =1}$ of $\mathcal{F}^{\rm d}$ comprising distributions fixed by the action of $\tau$ plays an important role in our approach and, in particular, the difficulty of explicitly characterizing $(1+\tau)\mathcal{F}^{\rm d}$ as a submodule of $\mathcal{F}^{{\rm d},\tau =1}$ often causes extra complications when considering $2$-primary aspects of problems (the quotient group $\mathcal{F}^{{\rm d},\tau =1}/(1+\tau)\mathcal{F}^{\rm d}$ is a vector space over the field of two elements of uncountably infinite dimension - see Remark \ref{tate coh rem}). We are able to resolve many, but not all, of these $2$-primary issues and, as a consequence, some of our results are slightly less complete than we would have liked.

In a little more detail, the main contents of this article is as follows. In \S\ref{basic prop sec} we recall some basic facts concerning  distributions, introduce various useful notions of `partial distribution', prove several technical results concerning distributions with totally positive values that are important for later arguments and discuss various explicit examples that give some idea of the difficulty of working with  distributions. In \S\ref{circ dist} we recall the well-known link between distributions and Euler systems and combine this with results of Greither to establish (in Theorem \ref{first step}) a concrete link between the values of distributions in the groups $\mathcal{F}^{\rm d}$ and $\mathcal{F}^{\rm c}$. In \S\ref{torsion section} we make a detailed study of distributions that are valued in roots of unity and prove (in Theorem
 \ref{reduction result}) an important reduction step in the proof of Theorem \ref{main result}. In \S\ref{pd section} we prove (in Theorem \ref{key real}) several key properties of distributions `of prime level' and then, in \S\ref{proofs section}, we combine the results of \S\ref{torsion section} and \S\ref{pd section} to prove Theorem \ref{main result} and then also justify the various remarks that are made after the above statement of this result. Finally, in \S\ref{divisible section} we shall use Theorem \ref{main result} to give an explicit criterion for the image of a distribution in $\mathcal{F}^{\rm d}_{\rm tf}/\mathcal{F}^{\rm c}$ to be $p$-divisible and also show that `divisible' distributions have many of the same properties as cyclotomic distributions.

Throughout the article we shall usually use exponential notation to indicate the action of a commutative ring $\Lambda$ on a multiplicative group $U$, so that the image of an element $u$ of $U$ under the action of an element $\lambda$ of $\Lambda$ is written as $u^\lambda$. However we caution the reader that, for typographic simplicity, we shall also occasionally write either $\lambda(u)$ or $\lambda\cdot u$ in place of $u^\lambda$.
\\

\section{Basic properties of distributions}\label{basic prop sec}

In this section we shall first recall some well-known properties of distributions. We next establish several further properties that will be useful  in later arguments and then end by discussing various explicit examples that help clarify the theory.

\subsection{}We first introduce notation and review some basic facts concerning distributions (that are in the main due to Coleman).

\subsubsection{}In the sequel we set $\mathbb{N}^\ast := \mathbb{N}\setminus \{1\}$. 

For each $n$ in $\mathbb{N}^*$ we fix a generator $\zeta_n$ of $\mu_n$ such that $\zeta_{mn}^m = \zeta_n$ for all $m$ and $n$ and for each function $f$ in $\mathcal{F}$ we set
\[ f(n) := f(\zeta_n).\]

Having made such a choice of elements of $\mu_\infty^*$, the Galois equivariance of functions in $\mathcal{F}^{\rm d}$ allows us to identify each such $f$ with a set of the form $\{f(n)\}_{n \in \mathbb{N}^*}$, where each $f(n)$ belongs to $\QQ(n)^\times$ and, taken together, these elements satisfy suitable relations.

Before stating these relations we introduce some further notation. We set
\[ G_n := \Gal(\QQ(n)/\QQ)\,\,\text{ and }\,\, R_n := \ZZ[G_n].\]
For each multiple $m$ of $n$ we then write $G^m_n$ for the subgroup  $\Gal(\QQ(m)/\QQ(n))$ of $G_m$ and $\pi^m_n$ for the ring homomorphism $R_m \to R_n$ induced by the natural projection $G_m\to G_n$. We also write ${\rm N}^m_n$ for the field-theoretic norm map $\QQ(m)^\times \to \QQ(n)^\times$.

We write $\tau$ for the element of $G_n$ induced by complex conjugation and note that the maximal totally real subfield $\QQ(n)^+$ of $\QQ(n)$ is equal to the set of elements fixed by $\tau$.

For each prime $\ell$ we fix an element $\sigma_\ell$ of $\Gal(\QQ^c/\QQ)$ that restricts to give the identify automorphism on $\QQ(\ell^a)$ for all natural numbers $a$ and to give the {\em inverse} of the Frobenius element at $\ell$ on $\Q(m)$ for every natural number $m$ that is prime to $\ell$.

Then a straightforward exercise shows that the defining property (\ref{defining prop}) of $f$ is equivalent to requiring that for each natural number $m$ and prime $\ell$ one has

\begin{equation}\label{basic facts} {\rm N}^{m\ell}_m(f(m\ell)) = \begin{cases} f(m), &\text{ if $\ell$ divides $m$,}\\
f(m)^{1-\sigma_\ell}, &\text{ if $\ell$ is prime to $m$.}\end{cases}\end{equation}

We write $E(n)$ and $E^p(n)$ for each prime $p$ for the group of global units, respectively global $p$-units, in $\QQ(n)$ and set

\[ E(n)' := \begin{cases} E^p(n), &\text{ if $n$ is a power of a prime $p$,}\\
                          E(n), &\text{ otherwise.}\end{cases}\]

Then the first distribution relation in (\ref{basic facts}) implies that for each $n$ in $\mathbb{N}^*$ one has
\begin{equation}\label{unit prop} f(n) \in E(n)' \end{equation}
(for details see, for example, \cite[Lem. 2.2]{Seo1}).

%

%
%

\subsubsection{}\label{distributions of level} For any subset $\Sigma$ of $\mathbb{N}^*$ we set $\mu^*_\Sigma := \bigcup_{m \in \Sigma}\mu^*_m$ and write $\mathcal{F}^{\rm d}_\Sigma$ for the multiplicative group of Galois equivariant maps from $\mu^*_\Sigma$ to $\QQ^{c,\times}$ that satisfy the relation (\ref{defining prop}) for all $\varepsilon$ in $\mu_\Sigma^*$ and all natural numbers $a$ for which there exists an element $\zeta$ of $\mu^*_\Sigma$ such that $\zeta^a = \varepsilon$.

We then write $\mathcal{F}^{\rm sd}_\Sigma$ for the subgroup of $\mathcal{F}^{\rm d}_\Sigma$ comprising functions that also satisfy the congruence relation (\ref{defining prop2}) for all coprime $n$ and $\ell$ for which $n$ and $n\cdot \ell$ both belong to $\Sigma$.

It is clear that each such group $\mathcal{F}^{\rm d}_\Sigma$ and $\mathcal{F}^{\rm sd}_\Sigma$ has a natural structure as $R$-module and we write $\iota_\Sigma$ for both of the homomorphisms of $R$-modules $\mathcal{F}^{\rm d} \to \mathcal{F}^{\rm d}_{\Sigma}$ and $\mathcal{F}^{\rm sd} \to \mathcal{F}^{\rm sd}_{\Sigma}$ that are obtained by restricting functions from $\mu_\infty^*$ to $\mu^*_\Sigma$.

In the special case that $\Sigma$ is equal to the subset $\mathbb{N}(m)$ of $\mathbb{N}^*$ comprising all multiples of a given natural number $m$, then we shall abbreviate $\mathcal{F}^{\rm d}_\Sigma$ and $\mathcal{F}^{\rm sd}_\Sigma$ to $\mathcal{F}^{\rm d}_{(m)}$ and
$\mathcal{F}^{\rm sd}_{(m)}$ and refer to functions in these sets as `distributions of level $m$' and `strict distributions of level $m$' respectively. In this case we shall also write $\iota_m$ in place of $\iota_{\Sigma}$.


\subsection{}\label{totally positive section} In this section we establish several useful results concerning distributions with totally positive values.

For each $m$ in $\mathbb{N}^*$ we write $\QQ(m)^+_{>0}$ for the multiplicative group of totally positive elements in $\QQ(m)^+$ and then define a torsion-free group by setting
\[ V(m) := E(m)'\cap \QQ(m)^+_{>0}.\]
%

For each subset $\Sigma$ of $\mathbb{N}^*$ we write $\mathcal{F}_\Sigma^{{\rm d},+}$ and $\mathcal{F}_\Sigma^{{\rm sd},+}$ for the $R$-submodules of $\mathcal{F}_\Sigma^{\rm d}$ and $\mathcal{F}_\Sigma^{\rm sd}$ comprising functions $f$ with $f(m) \in V(m)$ for all $m$ in $\Sigma$.

\begin{lemma}\label{useful 1} If $\Sigma$ is any cofinal subset of $\mathbb{N}^*$ that is closed under taking powers, then the natural restriction map $\iota_{\Sigma}: \mathcal{F}^{{\rm d},+} \to \mathcal{F}^{{\rm d},+}_\Sigma$ is injective.
\end{lemma}

\begin{proof} To argue by contradiction we assume $f$ is a non-trivial distribution with $\iota_\Sigma(f) = 1$. We fix $m$ in $\mathbb{N}^*$ with $f(m) \not = 1$ and a prime divisor $p$ of $m$. We then fix a natural number $n$ in $\Sigma\cap \mathbb{N}(m)$ and note that, since $\Sigma$ is closed under taking powers, for each $a\in \mathbb{N}$ there exists a natural number $n_a$ in $\Sigma\cap \mathbb{N}(mp^a)$ that has the same prime divisors as $n$.

In particular, if we write $\mathcal{P}$ for the (possibly empty) set of primes that divide $n$ but not $m$, and hence divide $n_a$ but not $mp^a$, then for each $a$ one has
\[ 1 = {\rm N}^{n_a}_{mp^a}(1) = {\rm N}^{n_a}_{mp^a}(f(n_a)) = f(mp^a)^{\prod_{\ell}(1-\sigma_\ell)}\]
where in the product $\ell$ runs over all primes in $\mathcal{P}$.

If $\mathcal{P}$ is empty, this implies $f(mp^a) = 1$ and hence also, since $p$ divides $m$, that $f(m) = {\rm N}^{mp}_m(f(mp)) = 1$, which is a contradiction.

On the other hand, if $\mathcal{P}$ is non-empty, then, since the elements $f(mp^a)$ for $a\ge 0$ form a norm-compatible family of elements of $V(mp^a)$, we obtain a contradiction to the assumption $f(m) \not= 1$ by successively applying Lemma \ref{key useful} below with $q$ taken to be each of the primes in $\mathcal{P}$.
\end{proof}

\begin{lemma}\label{key useful} Let $(x_a)_{a \ge 0}$ be a norm compatible family of elements of $V(mp^a)$ for which there exists a prime $q$ that does not divide $m$ and is such that $x_a^{1-\sigma_q} = 1$ for all $a$. Then one has $x_0 = 1$.
\end{lemma}

\begin{proof} Since the image of $\sigma_q$ in $\Gal(\QQ(mp^\infty)/\QQ)$ generates an open subgroup, each element $x_a$ belongs to $\QQ(mp^d)$ for a fixed natural number $d$.

Thus, for each integer $e >d$ one has
\[ x_0 = {\rm N}^{mp^e}_m(x_e) = {\rm N}^{mp^d}_m(x_e)^{p^{e-d}} \in V(m)^{p^{e-d}}.\]
Given this, the triviality of $x_0$ follows from the fact that the intersection of $V(m)^{p^{e-d}}$ over all $e > d$ is trivial.
\end{proof}

\begin{lemma}\label{useful 2} For every prime $\ell$ the endomorphism of $\mathcal{F}^{{\rm d},+}$ that sends each $f$ to $f^{1-\sigma_\ell}$ is injective. \end{lemma}

\begin{proof} Let $f$ be a non-trivial element of $\mathcal{F}^{{\rm d},+}$ and fix a prime $p$ different from $\ell$.

By applying Lemma \ref{useful 1} with $\Sigma = \mathbb{N}(p)$ we can fix $m$ in $\mathbb{N}(p)$ with $f(m) \not= 1$. Then Lemma \ref{key useful} implies that for some natural number $a$ one has
\[ f^{1-\sigma_\ell}(mp^a) = f(mp^a)^{1-\sigma_\ell} \not= 1.\]
This shows that the distribution $f^{1-\sigma_\ell}$ is non-trivial, as required.
\end{proof}

In the sequel we consider the quotient
\[ G_n^+ := \Gal(\QQ(n)^+/\QQ)\]
of $G_n$ and write $I_n$ for the annihilator in $\ZZ[G_n^+]$ of the element
\[ \varepsilon_n:= (1-\zeta_n)^{1+\tau} \in V(n).\]
of $V(n)$.

In the next result we shall give an explicit description of this ideal in terms of the decomposition behaviour in $\QQ(n)^+$ of prime divisors of $n$. However, before stating this result, we must introduce more notation.

We write $\#X$ for the cardinality of a finite set $X$. For a finite group $\Gamma$ we write $e_\Gamma$ for the idempotent $e_\Gamma := \#\Gamma^{-1}\cdot \sum_{\gamma \in \Gamma}\gamma$ of $\QQ[\Gamma]$.

We then obtain an idempotent of $\QQ[G_n^+]$ by setting
\begin{equation}\label{n idem} e_n := \begin{cases} e_{G_n^+} + \prod_{\ell \mid n}(1-e_{D_{n,\ell}}), &\text{if $n$ is a prime power,}\\
                        \prod_{\ell \mid n}(1-e_{D_{n,\ell}}), &\text{otherwise,}\end{cases}\end{equation}
where in the products $\ell$ runs over all prime divisors of $n$ and $D_{n,\ell}$ denotes the decomposition subgroup of $\ell$ in $G^+_n$.

For each homomorphism $\psi: G_n \to \QQ^{c,\times}$ we write $e_\psi$ for the primitive idempotent $\#G_n^{-1}\sum_{g \in G_n}\psi(g^{-1})g$ of $\QQ^c[G_n]$.

For any element $u$ of $V(n)$ we then define $e_\psi\cdot u$ by means of the inclusion $V(n) \subset \QQ^c\otimes_\ZZ V(n)$.

\begin{lemma}\label{useful 3} $I_n$ is equal to the set $\{ x\in R_n^+ \, \mid \, e_n\cdot x =0\}$.
 \end{lemma}

\begin{proof} An element $x$ of $\ZZ[G_n^+]$ belongs to $I_n$ if and only if $x = (1-e)\cdot x$ where the idempotent $e$ is the sum of $e_\psi$ over all homomorphisms $\psi: G_n^+ \to \QQ^{c,\times}$ for which $e_\psi\cdot \varepsilon_n \not= 0$. It is therefore enough to show that $e$ is equal to $e_n$.

If $\psi$ is the trivial homomorphism, then it is easy to check that $e_\psi\cdot \varepsilon_n \not= 0$ if and only if $n$ is a prime power.

On the other hand, if $\psi$ is not trivial, then the distribution relations (\ref{basic facts}) combine with the fundamental link  between cyclotomic elements and first derivatives of Dirichlet $L$-series (as discussed, for example, in \cite[Chap. 3, \S5]{tate}) to imply that
\[ -\frac{1}{2}\sum_{g \in G} \psi(g)^{-1}{\rm log}(\sigma_\infty(\varepsilon_n^g)) = L'(\psi^{-1},0)\prod_{\ell\in \mathcal{P}_\psi} (1-\psi(\sigma_\ell^{-1})).\]
Here $\sigma_\infty$ denotes the embedding $\QQ(n) \to \CC$ that sends $\zeta_{n}$ to $e^{2\pi i/n}$ and $\mathcal{P}_\psi$ the set of prime divisors of $n$ that do not divide the conductor of $\psi$.

In particular, since $L'(\psi^{-1},0)\not= 0$, this equality implies that $e_\psi\cdot \varepsilon_n \not= 0$ if and only if the element
$e_\psi\cdot \prod_{\ell \in \mathcal{P}_\psi}(1-e_{D_{n,\ell}}) = e_\psi\cdot \prod_{\ell \mid n}(1-e_{D_{n,\ell}})$ is non-zero.

By combining these facts with the explicit definition of $e_n$ one directly checks that $e = e_n$, as required. \end{proof}

\begin{example}\label{annihilator exam}{\em Lemma \ref{useful 3} can be used to extend the explicit computation made by the second author in \cite[Prop. 2.4]{Seo4} beyond the special cases that are considered there. For the moment we only record two easy examples.

\noindent{}(i) If $n$ is a power of a prime $p$, then $D_{n,p} = G_n^+$ so $e_n = 1$ and hence $I_n$ vanishes. 

\noindent{}(ii) If $n$ is not a prime power and there exists a prime divisor $p$ of $n$ such that $D_{n,p} \subseteq D_{n,\ell}$ for all prime divisors $\ell$ of $n$, then $e_n = 1-e_{D_{n,p}}$ and so $I_{n} = \ZZ[G_n^+]\cdot \sum_{g \in D_{n,p}}g$.
}\end{example}

\begin{remark}\label{more general ann}{\em Fix $f$ in $\mathcal{F}^{\rm d}$. Then for every $n$ in $\mathbb{N}^\ast$ the containment (\ref{unit prop}) implies $f(n)^{1+\tau}$ belongs to $V(n)$, whilst the relations (\ref{basic facts}) combine with the argument of Lemma \ref{useful 3} to imply that $f(n)^{1+\tau} = e_n\cdot f(n)^{1+\tau}$ in $\QQ\otimes_\ZZ V(n)$. In particular, since for every homomorphism $\psi: G_n^+ \to \QQ^{c,\times}$ with $e_\psi\cdot e_n \not= 0$ one has both $e_\psi\cdot \varepsilon_n \not= 0$ and  ${\rm dim}_{\QQ^c}\bigl(e_\psi(\QQ^c\otimes_\ZZ V(n)))= 1$ (by Dirichlet's Unit Theorem), there exists an element $j_{f,n}$ of $\QQ[G_n^+]e_n$ such that $f(n)^{1+\tau} = \varepsilon_n^{j_{f,n}}$ in $\QQ\otimes_\ZZ V(n)$.}
\end{remark}

\begin{remark}{\em The result of Lemma \ref{useful 3} concerns Galois relations between cyclotomic numbers in real abelian fields and so is, in principle, well-known (see, for example, the article of Solomon \cite{solomon} and the references contained therein). However, the precise nature of the result of Lemma \ref{useful 3} will play an important role in later arguments and so we have given a direct proof.}\end{remark}


\subsection{}In this section we collect together several general remarks concerning distributions.

\begin{remark}{\em If $f$ in $\mathcal{F}^{\rm d}$ is non-trivial, then for every prime $p$, there exists an $m$ in $\mathbb{N}(p)$ with $f(m) \not= 1$ (this follows from Lemma \ref{useful 1} with $\Sigma = \mathbb{N}(p)$). However, for any given finite set of prime numbers $\mathcal{P}$ there exists a non-trivial $f_\mathcal{P}$ in $\mathcal{F}^d$ with the property that $f_\mathcal{P}(\ell^m) = 1$ for all $\ell\in \mathcal{P}$ and $m\in \mathbb{N}$. In fact, for any non-trivial $f$ in $\mathcal{F}^{{\rm d},+}$ the distribution $f_\mathcal{P} := (\prod_{\ell \in \mathcal{P}}(1-\sigma_\ell))(f)$ has the required property and is non-trivial by Lemma \ref{useful 2}.} \end{remark}

\begin{remark}{\em If one fixes a prime $p$, then one can formulate a natural $p$-adic analogue of the conjectural equality (\ref{corrected coleman}) by defining $\mathcal{F}^{{\rm d},p}$ just as $\mathcal{F}^{{\rm d}}$ except that for each $n$ in $\mathbb{N}^*$ the condition (\ref{unit prop}) is replaced by $f(n) \in \ZZ_p\otimes_\ZZ E(n)'$ and then asking if the quotient of $\mathcal{F}^{{\rm d},p}$ by its torsion subgroup is generated  over the pro-$p$ completion $\widehat{R}^p$ of $R$ by $\Phi$? However, this question has a negative answer. For example, if one fixes any prime $\ell$ different from $p$, then the assignment
\[ f_\ell(n) := \begin{cases}  n^{-1}\otimes \ell, &\text{if $n$ is a power of $\ell$,}\\
                          1, &\text{otherwise,}\end{cases}\]
gives a well-defined, non-torsion, element $f_\ell$ of $\mathcal{F}^{{\rm d},p}$ and Lemma \ref{useful 1} (with $\Sigma = \mathbb{N}(p)$) implies that $f_\ell$ does not belong to $\widehat{R}^p\cdot \Phi$.}\end{remark}\
%
%

We finally give an example of a norm-compatible family of global units in the cyclotomic $\ZZ_p$-extension of a cyclotomic field that cannot arise as the restriction of a distribution.

To describe this we fix distinct odd primes $p$ and $q$ and for each $b$ in $\mathbb{N}$ we set $\Gamma _b := G_{qp^b}^+$ and we choose an element $T_b$ of $R$ that projects to the element $\sum_{g \in \Gamma_{b}}g$ of $\ZZ[\Gamma_{b}]$. For each $a$ in $\mathbb{N}^*$ we then set $\Pi_a := \sum_{b=1}^{b=a-1}T_b$.

\begin{lemma}\label{nc-nd exam}\
\begin{itemize}
\item[(i)] The family $((\varepsilon_{qp^a})^{\Pi_a})_{a \ge 2}$ is norm-compatible.
\item[(ii)] If $q-1$ is not divisible by the order of $q$ modulo $p$, then there is no $f$ in $\mathcal{F}^d$ such that $f(qp^a) = (\varepsilon_{qp^a})^{\Pi_a}$ for all $a \ge 2$.
\end{itemize}
\end{lemma}

\begin{proof} Claim (i) follows easily from the fact that for $a \ge 2$ one has $\varepsilon_{qp^a}^{T_a} = {\rm N}^{\QQ(qp^a)^+}_\QQ(\varepsilon_{qp^a}) = 1$.

To prove claim (ii) we argue by contradiction and thus assume that there exists $f$  in $\mathcal{F}^{\rm d}$ with $f(qp^a) = (\varepsilon_{qp^a})^{\Pi_a}$ for all $a \ge 2$. Then for all such $a$ one would have
\[ (\varepsilon_{p^a})^{\Pi_a(\sigma_q-1)} = {\rm N}^{qp^a}_{p^a}((\varepsilon_{qp^a})^{\Pi_a}) = {\rm N}^{qp^a}_{p^a}(f(qp^a)) = f(p^a)^{\sigma_q-1}\]
and hence there would exist a non-zero element $x_a$ of the fixed subfield, $L$ say, of $\bigcup_a\QQ(p^a)$ by $\sigma_q$ with
\begin{equation}\label{next} (\varepsilon_{p^a})^{\Pi_a} = x_a\cdot f(p^a).\end{equation}

Fix $t$ in $\mathbb{N}$ and $a$ in $\mathbb{N}$ with $L \subseteq \QQ(p^a)$. Then, by comparing (\ref{next}) to the image under ${\rm N}^{p^{a+t}}_{p^a}$ of the corresponding equality with $a$ replaced by $a+t$, and using the fact that the elements $\{f(p^a)\}_a$ are norm compatible, one finds that 
$x_{a+t}^{p^t} = (p^{q-1})^{\sum_{i=0}^{i=t-1}p^i}\cdot x_a$ and hence that

\[ p^t\cdot v_L(x_{a+t}) = (q-1)(\sum_{i=0}^{i=t-1}p^i)v_L(p) + v_L(x_a) = (q-1)v_L(p)\frac{p^t-1}{p-1} + v_L(x_a),\]
where $v_L(-)$ is the valuation on $L$ at the unique prime above $p$.

Now the stated assumption on $q-1$ implies that $(q-1)v_L(p) = r + s(p-1)$ for integers $r$ and $s$ that satisfy $0 < r < p-1$ and $s\ge 0$. In this case, therefore, one finds that for all sufficiently large $t$ the above equality implies that
\[ \begin{cases} -v_L(x_a) \equiv r\frac{p^t-1}{p-1} - s \,\,({\rm mod}\,\, p^t), &\text{if $v_L(x_a)\le 0$}\\
                  v_L(x_a) \equiv p^t- r\frac{p^t-1}{p-1} + s \,\,({\rm mod}\,\, p^t), &\text{if $v_L(x_a)> 0$.}\end{cases}\]
But, since $0 < r < p-1$, it is easily shown that no such integer $v_L(x_a)$ can exist.
\end{proof}

\section{Circular distributions and Euler systems}\label{circ dist} In this section we will establish a close link between the values of distributions in the groups $\mathcal{F}^{\rm d}$ and $\mathcal{F}^{\rm c}$.

The precise result is perhaps itself of some interest and will also later play a key role in the construction of the homomorphism $\kappa$ that occurs in Theorem \ref{main result}.

\subsection{}For an abelian group $A$ we write $A_{\rm tor}$ for its torsion subgroup and $A_{\rm tf}$ for the quotient of $A$ by $A_{\rm tor}$. For each prime $\ell$ we set $A_\ell := \ZZ_{\ell}\otimes_\ZZ A$.

For a natural number $n$ we set $\varphi(n) := \#G_n$. We also write $\ZZ_{(n)}$ and $\ZZ_{\{n\}}$ for the subrings of $\QQ$ that are obtained by intersecting  the $\ell$-localisations $\ZZ_{(\ell)}$ of $\ZZ$ over all primes $\ell$ that divide $n$ and over all primes $\ell$ that either divide $n$ or are coprime to $\varphi(n)$ respectively. We note, in particular, that $\ZZ_{\{n\}}$ is a subring of $\ZZ_{(n)}$.


We can now state the main result of this section.

\begin{theorem}\label{first step} Fix a natural number $m$. Then for each function $f$ in $\mathcal{F}^{\rm d}_{(m)}$ and each $n$ in $\mathbb{N}(m)$ there exists an element $r_{f,n}$ of $\ZZ_{(n)}\otimes_\ZZ R_n$ such that $f(n) = \Phi(n)^{r_{f,n}}$ in $\ZZ_{(n)}\otimes_\ZZ E(n)'$. If $f$ belongs to $\mathcal{F}^{\rm sd}_{(m)}$, then one can take each element $r_{f,n}$ to belong to $\ZZ_{\{n\}}\otimes_\ZZ R_n$.
\end{theorem}

This result is a considerable strengthening of the main results of the second author in \cite{Seo1} and \cite{Seo2} in that it considers the $\ell$-primary properties of the values of distributions at integers $n$ for which $\varphi(n)$ can be divisible by $\ell$.

In Corollary \ref{values-descent} we will also show Theorem \ref{first step} implies that any given distribution has `cyclotomic values' if and only if it has the same Galois descent properties as do cyclotomic units.

\subsection{}There are two differing notions of `Euler system' that are useful for us since they are respectively related to the groups of distributions $\mathcal{F}^{\rm sd}$ and $\mathcal{F}^{\rm d}$.

For the reader's convenience we now recall the basic facts concerning these notions.

\subsubsection{}\label{first euler} To describe the first we write $\mathcal{J}(m)$ for each natural number $m$ for the set of positive square-free integers that are only divisible by primes congruent to $1$ modulo $m$.

Then an Euler system over the field $\QQ(m)$ is defined by Rubin in \cite[\S1]{RL} to be a map
\[ \varepsilon: \mathcal{J}(m)\to \QQ^{c,\times}\]
that satisfies the following four conditions for each $r$ in $\mathcal{J}(m)$ and each prime divisor $\ell$ of $r$:

\begin{itemize}
\item[(ES$_1$)] $\varepsilon(r)$ belongs to $\QQ(mr)^\times$.
\item[(ES$_2$)] $\varepsilon(r)$ belongs to $E(mr)$ if $r> 1$.
\item[(ES$_3$)] ${\rm N}^{r}_{r/\ell}(\varepsilon(r)) = \varepsilon(r/\ell)^{\sigma_\ell-1}$.
\item[(ES$_4$)] $\varepsilon(r) \equiv \varepsilon(r/\ell)$ modulo all primes above $\ell$.
\end{itemize}

For each $f$ in $\mathcal{F}_{(m)}$ and each $\zeta$ in $\mu_\infty^*$, we define $\varepsilon_{f,\zeta}$ to be the function on $\mathcal{J}(m)$ obtained by setting
\[ \varepsilon_{f,\zeta}(m') := f(\zeta\cdot \prod_{\ell \mid m'}\zeta_\ell)\]
for each $m'$ in $\mathcal{J}(m)$.

Coleman observed that if $f$ belongs to $\mathcal{F}_{(m)}^{\rm sd}$ and $\zeta$ to $\mu_m$, then the second distribution relation in (\ref{basic facts}) combines with the congruence property (\ref{defining prop2}) to imply $\varepsilon_{f,\zeta}$ is an Euler system over $\QQ(m)$ in the above sense (cf. \cite[Lem. 4.1]{Seo1}).

In addition, the first relation in (\ref{basic facts}) implies that for any $m$ the elements $\varepsilon_{f,\zeta_{mp^i}}(1) = f(\zeta_{mp^i})$ form a norm-compatible family as $i$ varies over the natural numbers.

However, if $f$ belongs only to $\mathcal{F}_{(m)}^{\rm d}$ then for any $\zeta$ in $\mu_m$ the function $\varepsilon_{f,\zeta}$ need not be an Euler system in the above sense since a failure of $f$ to satisfy the congruence property (\ref{defining prop2}) means that $\varepsilon_{f,\zeta}$ need not satisfy the condition (ES$_4$).

For this reason we are led to consider an alternative definition of Euler system.

\subsubsection{}\label{second euler}For each natural number $m$ we write $\mathcal{R}(m)$ for the set of square-free products of primes that do not divide $m$.

To describe the second notion of Euler system we fix a prime $p$ and, for each $r$ in $\mathcal{R}(p)$, we write $\QQ(r)^p$ for the composite over all primes divisors $\ell$ of $r$ of the maximal totally real subextension of $\QQ(\ell)$ that has $p$-power degree over $\QQ$. We note, in particular, that for any natural number $m$ prime to $r$ the field $\QQ(mr)$ is an extension of $\QQ(m)\QQ(r)^p$ of degree prime to $p$.

We now fix a multiple $m$ of $p$. Then for each function $f$ in $\mathcal{F}_{(m)}$ and each natural number $t$ we define a function $\varepsilon_{f,t}$ on  $\mathcal{R}(m)$ by setting
\[ \varepsilon_{f,t}(r) := {\rm Norm}_{\QQ(p^tmr)/\QQ(p^tm)\QQ(r)^p}(f(p^tmr)).\]

Then, if $f$ belongs to $\mathcal{F}_{(m)}^{\rm d}$, the distribution relation (\ref{basic facts}) implies that, as $t$ varies over all natural numbers, the collection of functions $\varepsilon_{f,t}$ constitutes an Euler system in the sense defined by Rubin in
\cite[Def. 2.1.1 and Rem. 2.1.4]{R} for the data $K = \QQ(m), \mathcal{N} = \{p\}$, $T = \ZZ_p(1)$ and with $K_\infty$ taken to be the union of the fields $\QQ(p^tm)$ for $t \ge 1$. (See \cite[\S3.2]{R} for an explicit description of the cohomology groups that occur in this case.)

In particular, if $f$ belongs to $\mathcal{F}_{(p)}^{\rm d}$, then the argument of \cite[Cor. 4.8.1 and Exam. 4.8.2]{R} shows that for each $m$ in $\mathbb{N}(p)$, each $r$ in $\mathcal{R}(m)$ and each prime $\ell$ that does not divide $mr$ there exists an integer $t$ that is prime to $p$ and is such that
\begin{equation}\label{complicated} f(\zeta_\ell\cdot \zeta_{mr})^t \equiv f(\zeta_{mr})^t \,\,\text{ modulo all primes above $\ell$.}\end{equation}

This observation implies, in particular, that for any $f$ in $\mathcal{F}_{(p)}^{\rm d}$, any $m$ in $\mathbb{N}(p)$ and any $\zeta$ in $\mu_m$, the function $\varepsilon_{f,\zeta}$ defined in \S\ref{first euler} satisfies the conditions (ES$_1$), (ES$_2$), (ES$_3$) and a variant (ES$_4$)$_p$ of the condition (ES$_4$) in which the necessary congruences are only valid after projecting from the appropriate residue fields to their $p$-primary components.

\subsubsection{}\label{cyclo section} Finally we recall some basic facts concerning cyclotomic units.

For each $n$ in $\mathbb{N}^*$ the group of `circular numbers' in $\QQ(n)$ is defined by Sinnott \cite{sinnott} to be the subgroup of $\QQ(n)^\times$ given by
\[ C'(n) := \{(1-\zeta)^r: \zeta \in \mu_n^*, r \in R_n\}\]
and that the group $C(n):= C'(n)\cap E(n)$ of `cyclotomic units' has finite index in $E(n)$.

It is clear that for any $m$ in $\mathbb{N}(n)$ one has $E(n) \subseteq E(m)$ and $C(n) \subseteq C(m)$.
 The following observation of Gold and Kim concerning the induced map $E(n)/C(n) \to E(m)/C(m)$ will play a key role in our argument.

\begin{lemma}\label{gold kim inclusion} For each $n$ in $\mathbb{N}^*$ and each $m$ in $\mathbb{N}(n)$, one has $C(n) = H^0(G^m_n,C(m))$ and hence the natural map $E(n)/C(n) \to E(m)/C(m)$ is injective.\end{lemma}

\begin{proof} See \cite[Cor. 3]{goldkim}. \end{proof}

\subsection{} As an important preliminary to the proof of Theorem \ref{first step}, in this section we shall prove a technical result about inverse limits of cyclotomic units. To do this we fix a prime $p$ and a multiple $m$ of $p$.

We define the group of `unit-valued distributions of level $(p,m)$' to be the $R$-submodule  $\mathcal{F}_{(p,m)}^{\rm ud}$ of $\mathcal{F}^{\rm d}$ comprising maps $f$ with the property that for every non-negative integer $t$ one has $f(mp^t)\in E(mp^t)$. We note, in particular, that if $m$ is not a power of $p$, then (\ref{unit prop}) implies that $\mathcal{F}_{(p,m)}^{\rm ud} = \mathcal{F}^{\rm d}$.

For each non-negative integer $t$ we then write $\mathcal{F}^{\rm u}(mp^t)$ for the $R_{mp^t}$-submodule of $E(mp^t)$ that is generated by the set $\{f(\zeta): f\in \mathcal{F}_{(p,m)}^{\rm ud}, \zeta \in \mu_{mp^t}^*\}$.

With this definition, the inclusion $\mathcal{F}^{\rm c} \subseteq
\mathcal{F}^{\rm d}$ implies $C(mp^t)\subseteq \mathcal{F}^{\rm u}(mp^t)$ and, in addition, it is also clear that for each $t' \ge t$ the following two properties are satisfied
\begin{equation}\label{inclusion} \begin{cases} &\text{ the inclusion } E(mp^t) \subseteq E(mp^{t'})\,\text{ restricts to a map } \, \mathcal{F}^{\rm u}(mp^t) \subseteq  \mathcal{F}^{\rm u}(mp^{t'}),\\
&\text{ the norm map } {\rm N}^{mp^{t'}}_{mp^t} \text{ restricts to a map } \, \mathcal{F}^{\rm u}(mp^{t'}) \to \mathcal{F}^{\rm u}(mp^t).\end{cases}
\end{equation}

In the following result we set
\[ \mathcal{F}^{\rm u}(m)_p^\infty := \varprojlim_{a\ge 0} \mathcal{F}^{\rm u}(mp^a)_p \,\,\text{ and }\,\, C(m)_p^\infty := \varprojlim_{a\ge 0}C(mp^a)_p,\]
where in both cases  the transition maps are induced by the norms ${\rm N}^{mp^{a+1}}_{mp^a}$.

\begin{proposition}\label{bijection lemma} The natural inclusion map $C(m)^\infty_p \to \mathcal{F}^{\rm u}(m)^\infty_p$ is bijective.
\end{proposition}

\begin{proof} We claim first it is enough to show that the quotient group $Q := \mathcal{F}^{\rm u}(m)^\infty_p/C(m)^\infty_p$ is finite.

To see this we fix a topological generator $\gamma$ of $\Gal(\QQ(mp^\infty)/\QQ(m))$. Then, if $Q$ is finite, there is a natural number $t_0$ such that for any $t \ge t_0$ the element $\gamma^{p^t}$ acts trivially on $\mathcal{F}^{\rm u}(m)^\infty_p/C(m)^\infty_p$. For any $t_1 \ge t_0$ we can then consider the composite homomorphism
\[\theta^{t_1}_{t_0}: (\mathcal{F}^{\rm u}(mp^{t_1})/C(mp^{t_1}))_p \to (\mathcal{F}^{\rm u}(mp^{t_0})/C(mp^{t_0}))_p \to (\mathcal{F}^{\rm u}(mp^{t_1})/C(mp^{t_1}))_p,\]
where the first map is induced by field-theoretic norm and the second by the relevant case of the homomorphism in Lemma \ref{gold kim inclusion}. In particular, since the latter homomorphism is injective the composite sends each $x$ to $\sum_{a=0}^{a=t_1-t_0}(\gamma^{p^{t_0+a}})^m(x)$.

Thus, if $p^{t_1-t_0} \ge \#Q,$ then for any $x$ in the projection of $Q$ to
$(\mathcal{F}^{\rm u}(mp^{t_0})/C(mp^{t_0}))_p$ one has $\theta^{t_1}_{t_0}(x) = p^{t_1-t_0+1}\cdot x = 0$ and hence $Q$ vanishes, as claimed.

We next write $E(m)^\infty_p$ for the inverse limit of the groups $E(mp^t)_p$ for $t \ge 0$ with respect to the maps ${\rm N}^{mp^{t+1}}_{mp^t}$. We recall that the
 main result of Greither in \cite{greither2} (and the earlier work of Kuz'min in \cite{kuzmin}) imply that the quotient module $E(m)^\infty_p/C(m)^\infty_p$ is finitely generated over $\ZZ_p$.

This observation implies that $Q$ is also finitely generated over $\ZZ_p$ and hence is finite if and only if the space $Q[1/p]$ vanishes.

To study this space we write $m'$ for the maximal divisor of $m$ prime to $p$ and use the algebra $\Lambda := \ZZ_p[[{\rm Gal}(\QQ(m'p^\infty)/\QQ)]]$.

We note that any choice of topological generator of $\Gal(\QQ(m'p^\infty)/\QQ(m'p))$ induces a natural identification of algebras
\[ \Lambda[1/p] = \bigoplus_{\psi \in \Xi} \Lambda_\psi[1/p]\]
where $\Xi$ denotes a set of representatives of the $\Gal(\QQ_p^c/\QQ_p)$-conjugacy classes of homomorphisms $G_{m'p}\to \QQ_p^{c,\times}$ and we set $\Lambda_\psi := \ZZ_p[\psi][[T]].$

For any $\Lambda$-module $N$ there is a corresponding direct sum decomposition of $\Lambda[1/p]$-modules
 $N[1/p] = \bigoplus_{\psi \in \Xi} N_\psi[1/p]$ where we write $N_\psi := N\otimes_{\ZZ_p[G(m'p)]}\ZZ_p[\psi]$, regarded as a module over $\Lambda_\psi$ in the obvious way.

It is therefore enough to show that each module $Q_\psi[1/p]$ vanishes and to do this we use the natural exact sequence
\begin{equation}\label{tauto} 0 \to Q_\psi[1/p] \to (E(m)^\infty_p/C(m)^\infty_p)_\psi[1/p] \to (E(m)^\infty_p/\mathcal{F}^{\rm u}(m)^\infty_p)_\psi[1/p]\to 0.\end{equation}

If $\psi$ is an odd character of $G_{m'p}$, then all modules in this sequence vanish. This is because in this case each module is invariant under the action of $1-\tau$ and for each non-negative integer $i$ and each $u$ in $E(mp^i)$ the element $u^{2(1-\tau)}$ has finite order (by \cite[Th. 4.12]{wash}) and hence belongs to $C(mp^i)$.

To deal with the case that $\psi$ is even we write $X^p_{m}$ for the inverse limit $\varprojlim_i {\rm Cl}(\QQ(mp^i))_p$ with respect to the natural norm maps. In this case the result \cite[Th. 3.1]{greither} of Greither (see also the comment at the bottom of \cite[p. 120]{greither2}) shows the existence of an integer $a$ for which there is an equality of characteristic ideals
\begin{equation}\label{greither equality} {\rm char}_{\Lambda_\psi}((E(m)^\infty_p/C(m)^\infty_p)_\psi) = \pi_\psi^a \cdot {\rm char}_{\Lambda_\psi}(X^p_{m,\psi}),\end{equation}
where $\pi_\psi$ is a uniformizer of $\QQ_p(\psi)$.

We next note that for any $f\in \mathcal{F}_{(p,m)}^{\rm ud}$, any non-negative $t$ and any $\zeta \in \mu_{mp^t}^*$ the observations made at the end of \S\ref{second euler} imply that $f(\zeta)$ is equal to the value at $1$ of a unit-valued function $\varepsilon_{f,\zeta}$ from $\mathcal{J}(mp^t)$ to $\QQ^{c,\times}$ that satisfies the conditions (ES$_1$), (ES$_2$), (ES$_3$) and (ES$_4$)$_p$.

These functions $\varepsilon_{f,\zeta}$ may therefore fail to be an Euler system in the strict sense of \cite[\S1]{R}. However, this failure is immaterial to us since the only place that the condition (ES$_4$) is used in \cite{R} is in the argument of Kolyvagin that proves \cite[Prop. 2.4]{R} and this argument applies the map $\varphi_\ell$ that is constructed in \cite[Lem. 2.3]{R}.

In particular, if one takes the integer $M$ in loc cit. to be a power of $p$ (which is the relevant case for us), then $\varphi_\ell$ factors through the projection of $(\mathcal{O}_F/\ell\mathcal{O}_F)^\times$ to its maximal quotient of $p$-power order and hence it is sufficient to replace (ES$_4$) by the weaker condition (ES$_4$)$_p$.

Given this observation, the proof of \cite[Th. 3.1]{greither} shows that in $\Lambda_\psi$ the ideal
${\rm char}_{\Lambda_\psi}(X^p_{m,\psi})$ divides ${\rm char}_{\Lambda_\psi}((E(m)^\infty_p/\mathcal{F}^{\rm u}(m)_p^\infty)_\psi)$.

Combining this fact with the equality (\ref{greither equality}) one deduces that the characteristic ideal over $\Lambda_\psi$ of the kernel of the natural projection map $  (E(m)^\infty_p/C(m)^\infty_p)_\psi \to (E(m)^\infty_p/\mathcal{F}^{\rm u}(m)_p^\infty)_\psi$ is generated by a power of $\pi_\psi$ and hence that this kernel has finite exponent.

Taken in conjunction with the exact sequence (\ref{tauto}), this last fact implies that the module $Q_\psi[1/p]$ vanishes, as required to complete the proof. \end{proof}

\subsection{}\label{semi-local}We are now ready to prove Theorem \ref{first step}. To do this we fix $f$ in $\mathcal{F}^{\rm d}_{(m)}$ and $n$ in $\mathbb{N}(m)$.

Writing $D(n)$ for the $R_n$-module generated by $1-\zeta_n$, we need to prove that $f(n)$ belongs to $\ZZ_{(n)}\otimes_\ZZ D(n)$ and to  $\ZZ_{\{n\}}\otimes_\ZZ D(n)$ for each $f$ in $\mathcal{F}^{\rm sd}_{(m)}$.

Thus, since no prime divisor of $n$ is invertible in $\ZZ_{(n)}$, the proof of \cite[Th. 2.4]{Seo3} implies that it is actually enough to show  $f(n)$ belongs to $\ZZ_{(n)}\otimes_\ZZ C'(n)$, respectively to $\ZZ_{\{n\}}\otimes_\ZZ C'(n)$.

To do this we note at the outset that the element $-\zeta_n = (1-\zeta_n)^{1-\tau}$ of $C'(n)$ is a generator of the torsion subgroup $W_n$ of $\QQ(n)^\times$ except if $n = 2n'$ with $n'$ odd in which case one has $\QQ(n) = \QQ(n')$, $W_{n} = W_{n'}$, and $f(n) = f(n')^{1-\sigma_2}$ as a consequence of (\ref{basic facts}).

It is therefore enough to prove that for every $n$ the image of $f(n)$ in $V_n := \QQ\otimes_\ZZ \QQ(n)^\times$
belongs to $\ZZ_{n}\otimes_\ZZ C'(n)_{\rm tf}$ and to $\ZZ_{\{n\}}\otimes_\ZZ C'(n)_{\rm tf}$ if $f$ belongs to $\mathcal{F}^{\rm sd}_{(m)}$.

Now $\ZZ_{(n)}\otimes_\ZZ C'(n)_{\rm tf}$ and $\ZZ_{\{n\}}\otimes_\ZZ C'(n)_{\rm tf}$ are respectively equal to the intersections in $V_n$ of $\ZZ_{(p)}\otimes _\ZZ C'(n)_{\rm tf}$, as $p$ runs over all primes that divide $n$ and over all primes that either divide $n$ or are prime to $\varphi(n)$.

We are therefore reduced to proving that for each such $p$ one has in $E(n)'_{{\rm tf},p}$ a containment
\begin{equation}\label{needed contain} f(n) \in C'(n)_{{\rm tf},p}.\end{equation}
Since this result is obvious when $n=2$ we will also assume in the sequel that $n >2$ and hence that $\varphi(n)$ is even.

\subsubsection{}\label{gras section}In the case that $f$ belongs to $\mathcal{F}^{\rm sd}_{(m)}$ and $p$ is prime to both $n$ and $\varphi(n)$ (and hence is odd), the containment (\ref{needed contain}) is  well-known.

To explain this we write $G^\ast_n$ for the group $\Hom(G_n,\QQ_p^{c,\times})$ and we fix a finite extension of $\QQ_p$, with valuation ring $\mathcal{O}$, that contains the image of all homomorphisms in $G^\ast_n$.

Then for any $G_n$-module $X$ there is a direct sum decomposition of $\mathcal{O}[G_n]$-modules
\[ \mathcal{O}\otimes_\ZZ X = \bigoplus_{\chi\in G^\ast_n} X^\chi\]
where $X^\chi$ denotes the  $\chi$-isotypic component $\{x \in \mathcal{O}\otimes_\ZZ X: g(x) = \chi(g)\cdot x \, \text{ for all } g \in G_n\}$.

If $\chi(\tau)=-1$, then $(E(n)'_{{\rm tf}})^\chi$ vanishes, whilst if $\chi$ is trivial it is easily checked that $(E(n)'_{{\rm tf}})^\chi \subseteq C'(n)_{\rm tf}^\chi$.

It is therefore enough to prove that for each non-trivial $\chi$ with $\chi(\tau) = 1$ the image of $f(n)$ in $(E(n)')^\chi$ is contained in $C'(n)^\chi$. We therefore now fix such a $\chi$.

We write $U_n$ for the unit group of the field $\QQ(n)$ and $\mathfrak{E}_n$ for the $R_n$-submodule of $\QQ(n)^\times$ that is generated by the values at $1$ of all Euler systems over $\QQ(n)$ of the form $\varepsilon_{h,\zeta}$, with $h$ in $\mathcal{F}_{(n)}^{\rm sd}$ and $\zeta$ in $\mu_n$, as discussed in \S\ref{first euler}.

Then, since both $f(n)$ and $C'(n)$ are contained in $\mathfrak{E}_n$ it is enough to prove that $\mathfrak{E}_n^\chi$ is contained in $C'(n)^\chi.$

Now, the $\mathcal{O}$-modules $\mathfrak{E}_n^\chi$, $U_n^\chi$ and $C'(n)^\chi$ are all free of rank one and, since $\chi$ is non-trivial, one has $\mathfrak{E}_n^\chi \subseteq U_n^\chi$.

The required inclusion is therefore a direct consequence of the fact that
\[ {\rm char}_\mathcal{O}\bigl(U_n^\chi/\mathfrak{E}_n^\chi\bigr) \subseteq {\rm char}_\mathcal{O}\bigl({\rm Cl}(\QQ(n))^\chi) = {\rm char}_\mathcal{O}\bigl(U_n^\chi/C'(n)^\chi\bigr).\]
Here we write ${\rm char}_\mathcal{O}(X)$ for the order ideal of a finite $\mathcal{O}$-module $X$ and ${\rm Cl}(\QQ(n))$ for the ideal class group of $\QQ(n)$. In addition, the displayed inclusion follows from the
 argument used by Rubin to prove \cite[Th. 3.2]{rubin-crelle} and the fact that $\QQ(n) \cap \QQ(p) = \QQ$ since we are assuming that $p$ is prime to $n$, and the displayed equality is true as a consequence of the known validity of the Gras Conjecture (which was shown by Greenberg in \cite{green} to follow from the Main Conjecture of Iwasawa Theory for abelian fields, as subsequently proved, for example, by Mazur and Wiles).

\subsubsection{}It is now enough to prove (\ref{needed contain}) in the case that $f$ belongs to $\mathcal{F}^{\rm d}$ (but not necessarily to $\mathcal{F}^{\rm sd}$) and that $p$ divides $n$.

We assume first that $n$ is not a power of $p$ and so is divisible by at least two primes. In this case the restriction of $f$ to $\mathbb{N}(n)$ is unit-valued (by (\ref{unit prop})) and so $\{f(p^in)\}_{i \ge 0}$ belongs to the limit $\mathcal{F}^{\rm u}(n)_p^\infty$ that occurs in Proposition  \ref{bijection lemma}.

The latter result therefore implies that $\{f(p^in)\}_{i \ge 0}$ belongs to $C(n)^\infty_p$ and hence that $f(n)$ belongs to $C(n)_p$, as required.

If now $n$ is a power of $p$, then we set $a_p := 2$ if $p = 2$ and $a_p = 0$ if $p\not= 2$. We then fix an element $\gamma$ of ${\rm Gal}(\QQ^c/\QQ(p^{a_p}))$ that restricts to give a generator of $\Gal(\QQ(p^\infty)/\QQ(p^{a_p}))$ and define a function $f_\gamma$ on $\mathbb{N}(n)$ by setting $f_\gamma(n') := f(n')^{\gamma-1}$ for all multiples $n'$ of $n$.

Then this function $f_\gamma$ belongs to $\mathcal{F}_{(p,m)}^{\rm ud}$ and so the same argument as above shows that $f_\gamma(n) = f(n)^{\gamma_n-1}$ belongs to $C(n)_p$, where we write $\gamma_n$ for the image of $\gamma$ in $G_n$.

Now, since $n$ is a power of $p$, this implies the existence of an element $r_n$ of $R_{n,p}$ with
\[ f(n)^{\gamma_n-1} = (1-\zeta_n)^{r_n}.\]
Further, since the norm to $\QQ(p^{a_p})$ of this element is trivial, one has $r_n = (\gamma_n-1)\cdot r_n'$ for some element $r_n'$ of $R_{n,p}$ and it is enough to show that the element
\[ x:= f(n)/(1-\zeta_n)^{r'_n}\]
belongs to $C(n)_p$. Note also that $x$ belongs to the $p$-completion of the group of $p$-units in $\QQ(p^{a_p})$ as a consequence of (\ref{unit prop}).

In particular, if $p\not= 2$, then $x$ has the form $\pm p^b$ for some $b$ in $\ZZ_p$ and it is enough to note that both $-1 = (1-\zeta_n)^{n(1-\tau)}$ and $p=(1-\zeta_n)^{\sum_{g \in G_n}g}$ belong to $C(n)$.

If, lastly, $p = 2$, then there are elements $a$ and $b$ of $\ZZ_p$ such that $x = i^a(1-i)^b$. To deal with this case we can also assume, without loss of generality, that $\zeta_4 =i$.

Then if $n=4$, it is enough to note that $i = (1-\zeta_n)^{\tau-1}$ and $1-i = 1-\zeta_n$ belong to $C(4)$.  Similarly, if $n$ is divisible by $8$, then the containment $x \in C(n)_p$ follows from the fact that $i = (1-\zeta_n)^{n(1-\tau)/4}$ and $(1-\zeta_n)^{\sum_{g \in G_n}g} = i^c(1-i)$ for some integer $c$.

This completes the proof of Theorem \ref{first step}.

\subsection{} Finally, we show Theorem \ref{first step} implies that any given distribution in $\mathcal{F}^{\rm d}$ has `cyclotomic values' if and only if it has the same Galois descent properties as Lemma \ref{gold kim inclusion} implies for cyclotomic units.

\begin{corollary}\label{values-descent} Fix a distribution $f$ in $\mathcal{F}^{\rm d}$. For $n$ in $\mathbb{N}^*$ write $C_f(n)$ for the $R_n$-submodule of $E(n)$ generated by $C(n)$ together with $(\Phi^{-v_{f,n}}f)(n)$ where $v_{f,n} = 0$ unless $n = p^d$ for some prime $p$ in which case $v_{f,n}$ is the valuation of $f(n)$ at the unique $p$-adic place of $\QQ(n)$.

Then for every $n$ in $\mathbb{N}^*$ there exists an element $r'_{f,n}$ of $R_n$ such that $f(n) = \Phi(n)^{r'_{f,n}}$ if and only if for all $m$ in $\mathbb{N}^*$ and all $m'$ in $\mathbb{N}(m)$ one has $C_f(m) = H^0(G^{m'}_m,C_f(m'))$. \end{corollary}

\begin{proof} Necessity of the given conditions is an immediate consequence of Lemma \ref{gold kim inclusion} since if $f(m) = \Phi(m)^{r'_{f,m}}$ for some $r'_{f,m}$ in $R_m$, then $C_f(m) = C(m)$.

For each $n$ we set $f_n := \Phi^{-v_{f,n}}f$. Then to show sufficiency of the stated conditions it suffices, by virtue of the observation made at the beginning of \S\ref{semi-local}, to check that for each prime $\ell$ these conditions imply that $f_n(n)$ belongs to
the $\ell$-adic completion $C(n)_\ell$ of $C(n)$.

It is thus enough to note that these conditions imply that
\[ f_n(n)\in C_f(n)_\ell = H^0(G^{n\ell}_n,C_f(n\ell)_\ell) = H^0(G^{n\ell}_n,C(n\ell)_\ell) = C(n)_\ell\]
where the first equality holds by assumption, the second follows from Theorem \ref{first step} and the third from Lemma \ref{gold kim inclusion}.\end{proof}

\begin{remark}\label{factorial rem}{\em The set $\Sigma := \{n \in \mathbb{N}: \ZZ_{\{n\}}= \ZZ\}$ is cofinal in $\mathbb{N}$ (since it contains $d!$ for all $d$ in $\mathbb{N}$) and is also closed under taking powers. We may therefore apply Lemma \ref{useful 1} to this set to deduce that any $f$  in $\mathcal{F}^{\rm d}$ is uniquely determined by its values at integers in $\Sigma$. This suggests it is possible that, rather than relying on possible Galois descent properties as in Corollary \ref{values-descent}, Theorem \ref{first step} could itself directly imply that for every $f$ in $\mathcal{F}^{\rm sd}$ and every $m$ in $\mathbb{N}^*$ there exists an element $r'_{f,n}$ of $R_n$ with $f(n) = \Phi(n)^{r'_{f,n}}$.  }\end{remark}

\begin{remark}\label{soogil rem}{\em The second author would like to take this opportunity to point out that the proof of the main result of \cite{Seo1} is incorrect as given and further that it seems the argument in loc. cit. can only be corrected either by assuming the sort of Galois descent property on distributions that arises in Corollary \ref{values-descent}, or by inverting all primes that divide $\varphi(nt)$ but not $nt$.
 The point is that \cite[Lem. 5.4(i)]{Seo1} claims that the given inverse limits are equal if and only if $\mathfrak{E}_{p,s}(\mu_{p^rs}) = \mathcal{C}_{p^rs}$ for all $r$ and this may not be true since not all elements of $\mathfrak{E}_{p,s}(\mu_{p^rs})$ must be universal norms in $\QQ(p^{\infty}s)/\QQ(p^rs)$.
A similar issue arises with aspects of the argument used to prove \cite[Th. 2.4]{Seo2} but, given the assumptions in loc. cit., the problem in that case can be avoided by using the approach of \S\ref{gras section}. The second author would like to thank the first author for noticing this.}
\end{remark}

\section{Torsion-valued distributions}\label{torsion section}

In the sequel we shall say that a distribution $f$ in $\calF^{\rm d}$ is `torsion-valued' if $f(n)$ has finite order for every $n$ in $\mathbb{N}^*$ and we write $\calF^{\rm d}_{\rm tv}$ for the $R$-submodule of $\calF^{\rm d}$ comprising all such distributions.
 In particular, it is clear that $\calF^{\rm d}_{\rm tv}$ contains the torsion subgroup $\calF^{\rm d}_{\rm tor}$ of $\calF^{\rm d}$.


We shall now make a detailed study of torsion-valued distributions in order to prove the following result.

In this result, and the sequel, we regard the union $\QQ^{\rm ab}$ of $\QQ(n)$ over $n$ in $\mathbb{N}^*$ as a subfield of $\CC$ and write $\tau$ for the element of $\Gal(\QQ^{\rm ab}/\QQ)$ induced by complex conjugation.

We also use the $R$-module $\widehat{\ZZ}(1)$ defined just prior to the statement of Theorem \ref{main result}.

\begin{theorem}\label{reduction result}\
\begin{itemize}
\item[(i)] $\mathcal{F}_{\rm tv}^{\rm d}$ is equal to $\mathcal{F}_{\rm tor}^{\rm d}+ (1-\tau)\mathcal{F}^c$ and is also the kernel of the endomorphism of $\mathcal{F}^{\rm d}$ that is induced by multiplication by $1+\tau$.
\item[(ii)] There exist canonical isomorphisms of $R$-modules
 \[ \mathcal{F}_{\rm tv}^{\rm d}/\mathcal{F}^{\rm d}_{\rm tor}\cong (1-\tau)\mathcal{F}^{\rm c} \cong \widehat{\ZZ}(1).\]
The first of these is induced by the equality $\mathcal{F}_{\rm tv}^{\rm d} =\mathcal{F}_{\rm tor}^{\rm d}+ (1-\tau)\mathcal{F}^c$ in claim (i) and the second sends $\Phi^{1-\tau}$ to the topological generator $(\zeta_m)_m$ of $\widehat{\ZZ}(1)$.
\item[(iii)] If $\mathcal{F}^*$ denotes either $\mathcal{F}^{\rm d}$ or $\mathcal{F}^{\rm sd}$, then the $R$-module homomorphism
\[ \mathcal{F}^*_{\rm tf}/\mathcal{F}^{\rm c} \to (1+\tau)\mathcal{F}^*/(1+\tau)\mathcal{F}^{{\rm c}}\]
that is induced by multiplication by $1+\tau$ is bijective.
\end{itemize}
\end{theorem}

This result will play  a key role in the proof of Theorem \ref{main result}. After establishing several preliminary results concerning various inverse limits, it will be proved in \S\ref{proof of torsion}.

\subsection{}In the sequel, for each $n$ in $\mathbb{N}^*$ we write $\mathcal{T}_n$ for the annihilator of $-\zeta_n$ in $R_n$ and set
\[ \mathcal{T}^\ast_n := \begin{cases} \mathcal{T}_{n}, &\text{ if $n$ is even}\\
    \mathcal{T}_{2n}, &\text{ if $n$ is odd.}\end{cases}\]
We note, in particular, that if $n$ is odd, then $\mathcal{T}_n$ is a submodule of $\mathcal{T}^\ast_n$ of index $2$ (and that $n \in \mathcal{T}^\ast_n\setminus \mathcal{T}_n$).

\begin{lemma}\label{surjective lemma}\

\begin{itemize}
\item[(i)] For any natural number $n$ and prime $\ell$, one has
 $\pi^{n\ell}_{n}({\rm Ann}_{R_{n\ell}}(\mu_{n\ell})) \subseteq {\rm Ann}_{R_n}(\mu_n)$, with equality unless $\ell = 2$ and $n$ is odd.
\item[(ii)] For any $n$ in $\mathbb{N}^*$ and any $m$ in $\mathbb{N}(n)$ one has
 $\pi^m_{n}(\mathcal{T}^*_m) \subseteq \mathcal{T}^*_n$, with equality if $n\in \mathbb{N}(4)$.
\item[(iii)] The sequence
\[ 0 \to \varprojlim_n\mathcal{T}^*_n(1-\tau) \to R(1-\tau) \to \varprojlim_n R_n(1-\tau)/\mathcal{T}_n^*(1-\tau)\to 0\]
is exact, where the inverse limits are taken with respect to the transition maps induced by $\pi^m_n$, the second arrow denotes the natural inclusion and the third the natural diagonal map.
\end{itemize}
 \end{lemma}

\begin{proof} To prove claim (i) we note first that if $\ell = 2$ and $n$ is odd, then $\QQ(n\ell) = \QQ(n)$ whilst $\mu_n = (\mu_{n\ell})^2$ and so ${\rm Ann}_{R_{n\ell}}(\mu_{n\ell})$ is a subgroup of ${\rm Ann}_{R_n}(\mu_n)$ of index $2$.

To deal with the general case we set $m:= n\ell$ and consider the exact commutative diagram of $R_m$-modules
\[ \begin{CD}
0 @> >> {\rm Ann}_{R_m}(\mu_m) @> >> R_m @> \theta_m >> \mu_m @> >> 0\\
& & @V VV @V \pi^m_{n} VV @VV x\mapsto x^{\ell} V \\
0 @> >> {\rm Ann}_{R_n}(\mu_n) @> >> R_n @> \theta_n >> \mu_n @> >> 0,\end{CD}\]
where $\theta_m(1) = \zeta_m$ and $\theta_n(1) = \zeta_n$.

Then $\pi^m_n$ is surjective and so, by applying the Snake Lemma to this diagram, we find it is enough to prove that $\theta_m(\ker(\pi^m_{n})) = \mu_{\ell}$.

To do this we note $\ker(\pi^m_{n})$ is generated over $R_m$ by the set $\{\sigma-1: \sigma \in G^m_n\}$ and that an element $\sigma$ of $G_m$ belongs to $G^m_n$ if and only if $\sigma(\zeta_m) = \zeta_m^{a_\sigma}$ for an integer $a_\sigma$ with $a_\sigma \equiv 1$ (mod $n$).

In particular, if $\ell$ divides $n$, then there exists an element $\sigma$ of $G_n^m$ with $a_\sigma = 1 + n$ and $\theta_m(\sigma-1)$ is equal to the generator $\zeta_m^{n} = \zeta_\ell$ of $\mu_\ell$.

If $\ell$ is prime to $n$, then the first observation allows us to assume that $\ell$ is odd. In this case there are at least $\ell-2$ integers $a$ with $1\le a < \ell$  for which $1+a\cdot n$ is prime to $m$ and, for any such $a$, the corresponding element $\sigma$ of $G^m_n$ is such that $\theta_m(\sigma-1)= \zeta_m^{a\cdot n} = \zeta_\ell^a$ is a generator of $\mu_\ell$. This proves claim (i).

To prove claim (ii) we can reduce to the case $m = n\ell$ with $\ell$ prime. Then, after noting that for each natural number $t$ one has
\[ \mathcal{T}^*_t := \begin{cases} {\rm Ann}_{R_t}(\mu_{t/2}), &\text{ if $t \equiv 2$ (mod $4$),}\\
 {\rm Ann}_{R_t}(\mu_t), &\text{ otherwise,}\\
                                    \end{cases}\]
the claimed result follows directly from claim (i).
%
%
%

To prove claim (iii) we use the tautological short exact sequences
\[ 0 \to \mathcal{T}^*_n(1-\tau) \to R_n(1-\tau) \to R_n(1-\tau)/\mathcal{T}^*_n(1-\tau) \to 0.\]
As $n$ varies these sequences are compatible with the transition morphisms that are induced by $\pi^m_n$ (and the inclusions of claim (ii)). Hence, by passing to the inverse limit we obtain an exact sequence of $R$-modules
\begin{equation*}\label{tvd es} {\varprojlim}_n\mathcal{T}^*_n(1-\tau) \to R(1-\tau) \to {\varprojlim}_n R_n(1-\tau)/\mathcal{T}^*_n(1-\tau)\to {\varprojlim}_n^1\mathcal{T}^*_n(1-\tau),\end{equation*}
where we have identified $\varprojlim_n R_n(1-\tau)$ with $R(1-\tau)$ in the obvious way.

To deduce claim (iii) it thus suffices to note  that the derived limit $\varprojlim_n^1\mathcal{T}^*_n(1-\tau)$ vanishes since it can be computed by restricting to the cofinal subset $\mathbb{N}(4)$ for which claim (ii) implies that the transition maps $\mathcal{T}^*_m(1-\tau) \to \mathcal{T}^*_n(1-\tau)$ are surjective.
\end{proof}

\subsection{}In this section we consider functions $f$ in $\mathcal{F}^{\rm d}$ with the property that for all $n$ in $\mathbb{N}^*$ one has
\begin{equation}\label{key point} f(n)^{1+\tau} = 1.\end{equation}

We recall that $\QQ(n)^+$ denotes the maximal totally real subfield of $\QQ(n)$ and write $E(n)^+$ for the group of algebraic units $E(n)\cap \QQ(n)^+$ in $\QQ(n)^+$. We also recall that $W_n$ denotes the torsion subgroup of $E(n)$.

\begin{lemma}\label{reduction to torsion lemma} If $f$  is any distribution with property (\ref{key point}), then for every $n$ in $\mathbb{N}^*$ the element $f(n)$ belongs to the group $\langle -\zeta_n\rangle$ generated by $-\zeta_n$.

In particular, a distribution $f$ belongs to $\mathcal{F}_{\rm tv}^{\rm d}$ if and only it has property (\ref{key point}).
\end{lemma}

\begin{proof} It is clearly enough to prove the first assertion and to do this we first consider the case that $n = p^d$ for some prime $p$.

In this case the unique prime ideal of $E(n)$ above $p$ is stable under the action of $\tau$ and so the given equality
(\ref{key point}) combines with the containment (\ref{unit prop}) to imply that $f(n)$ belongs to $E(n)$.

From the result of \cite[Cor. 4.13]{wash} we can therefore deduce that $f(n)$ has the form $w\cdot u$ with $w$ in $W_n$ and $u$ an element of $E(n)^+$ that is either trivial or has infinite order. Given this, (\ref{key point}) implies that
\[ 1 = f(n)^{1+\tau} = w^{1+\tau}u^{1+\tau} = u^2\]
and hence that $u=1$ and so $f (n)\in W_n$.

Now if $n\not= 2$, then $W_n = \langle -\zeta_n\rangle$ and so we obtain the required claim. In addition, if $n =2$, then $f(2) = {\rm N}^4_2(f(4)) \in {\rm N}^4_2(\mu_4) = \{1\} = \langle -\zeta_2\rangle$.

We can therefore assume that $n$ is not a prime power. In this case (\ref{unit prop}) implies directly that $f(n)$ belongs to $E(n)$ and then the argument of \cite[Cor. 4.13]{wash} implies $f(n)$ can be written as $(1-\zeta_n)^awu$ with $a\in \{0,1\}$, $w\in W_n$ and $u\in E(n)^+$. Then the equality (\ref{key point}) implies that
\[ 1 = (1-\zeta_n)^{a(1+\tau)}u^2 = \varepsilon_n^au^2.\]

This implies that $\varepsilon_n^a$, and hence also
\[ (-\zeta_n)^a = (1-\zeta_n)^{a(1-\tau)} = (1-\zeta_n)^{a(1+\tau)}(1-\zeta_n)^{-2a\tau} = \varepsilon_n^a(1-\zeta_n^{-1})^{-2a},\]
is a square in $E(n)$. Since the argument of loc. cit. implies $-\zeta_n$ is not a square in $E(n)$ we deduce that $a=0$.

By using the same argument as above we can then deduce that $f(n)= wu$ must belong to $W_n$.

If $n \not\equiv 2$ (mod $4$), then $W_n = \langle -\zeta_n\rangle$ and we are done. In addition, if $n = 2n'$ with $n'$ odd, then
\[ f(n) = f(n')^{1-\sigma_2} \in \langle (-\zeta_{n'})^{1-\sigma_2}\rangle = \langle\zeta_{n'}\rangle = \langle-\zeta_n\rangle,\]
as required.\end{proof}

If $f$ is any distribution with property (\ref{key point}), then Lemma \ref{reduction to torsion lemma} implies that for each $n\in \mathbb{N}^*$ there exists an element $r'_{n} = r'_{f,n}$ of $R_n$ that is well-defined modulo $\mathcal{T}_n$ and is such that
\begin{equation}\label{def eq} f(n) = (1-\zeta_n)^{r'_n(1-\tau)}.\end{equation}

In particular, since $\mathcal{T}_n\subseteq \mathcal{T}_n^*$, the image $r_n = r_{n,f}$ in $R_n/\mathcal{T}_n^*$ of any such element $r'_n$ is uniquely determined by $f$ and $n$.

\begin{proposition}\label{inverse limit prop} For any $f$ in $\mathcal{F}^{\rm d}$ with property (\ref{key point}) the element $(r_n(1-\tau))_{n\in \mathbb{N}^\ast}$ belongs to the limit $\varprojlim_n R_n(1-\tau)/\mathcal{T}_n^*(1-\tau)$ that occurs in Lemma \ref{surjective lemma}(iii).
\end{proposition}

\begin{proof} It suffices to show that for each $n \in \mathbb{N}^*$ and each prime $\ell$ one has
\[ \pi^{n\ell}_n(r'_{n\ell}) \equiv r'_{n}\,\, \text{ modulo }\,\, \mathcal{T}^*_n.\]

If $\ell$ divides $n$, then this is true since the first equality in (\ref{basic facts}) implies that
\begin{multline*} (-\zeta_n)^{r'_{n\ell}} = (1-\zeta_n)^{r'_{n\ell}(1-\tau)} = ({\rm N}^{n\ell}_n(1-\zeta_{n\ell}))^{r'_{n\ell}(1-\tau)}\\  =
{\rm N}^{n\ell}_{n}(f(n\ell))= f(n)
 = (1-\zeta_{n})^{r'_{n}(1-\tau)} = (-\zeta_n)^{r_n'}.\end{multline*}

If $n$ is odd and $\ell = 2$, then $\pi^{n\ell}_n$ is the identity map and so it is enough to show that $r'_{n\ell} \equiv r'_{n}$ (mod $\mathcal{T}^\ast_{n}$). But in this case one has $\zeta^2_{n\ell} = \zeta_n = (\zeta_{n}^{\sigma_2})^2$ so that $-\zeta_{n\ell} = \zeta_{n}^{\sigma_2}$ and hence (\ref{basic facts}) implies
\begin{multline*} (\zeta_{n})^{r'_{n\ell}\cdot\sigma_2} = (-\zeta_{n\ell})^{r'_{n\ell}} = (1-\zeta_{n\ell})^{r'_{n\ell}(1-\tau)}  = f(n\ell)\\ = f(n)^{1-\sigma_2}
 = (1-\zeta_{n})^{r'_{n}(1-\tau)(1-\sigma_2)} = (-\zeta_{n})^{r'_{n}(1-\sigma_2)} = (\zeta_{n})^{r'_{n}\cdot\sigma_2}.\end{multline*}
This shows that $r'_{n\ell} \equiv r'_{n}$ (mod $\mathcal{T}^*_{n}$), as required.

Finally, to deal with the case that $\ell$ is odd and prime to $n$ we fix a prime divisor $q$ of $n$. We note first that, as the natural number $b$ varies, the equality (\ref{basic facts}) implies that the elements
\[ (-\zeta_{nq^b})^{r'_{nq^b}} = (1-\zeta_{nq^b})^{r'_{nq^b}(1-\tau)} = f(nq^b)\]
and
\[ (-\zeta_{n\ell q^b})^{r'_{n\ell q^b}} = (1-\zeta_{n\ell q^b})^{r'_{n\ell q^b}(1-\tau)} = f(n\ell q^b)\]
form elements $c$ and $c'$ of $W_\infty := \varprojlim_{b \ge 1}W_{n\ell q^b}$, where the limit is taken with respect to the norms ${\rm N}^{n\ell q^{b'}}_{n\ell q^b}$ for $b' \ge b$.

In addition, (\ref{basic facts}) also implies that for each $b$ one has
%

%
\begin{multline*}  (-\zeta_{nq^b})^{r'_{nq^b}(\ell-1)} = (1-\zeta_{nq^b})^{(1-\tau)r'_{nq^b}(1-\sigma_\ell)\sigma^{-1}_\ell}  = f(nq^b)^{(1-\sigma_\ell)\sigma^{-1}_\ell}\\ = {\rm N}^{n\ell q^b}_{n q^b}(f(n\ell q^b))^{\sigma^{-1}_\ell}
 = ({\rm N}^{n\ell q^b}_{n q^b}(1-\zeta_{n\ell q^b}))^{r'_{n\ell q^b}(1-\tau)\sigma^{-1}_\ell}\\  = (1-\zeta_{nq^b})^{(\sigma^{-1}_\ell-1)r'_{n\ell q^b}(1-\tau)} = (-\zeta_{nq^b})^{r'_{n\ell q^b}(\ell-1)}.\end{multline*}
This implies that the element $c^{-1}c'$ of $W_\infty$ is annihilated by raising to the power $\ell-1$. Since the maximal pro-$q$ quotient of $W_\infty$ is torsion-free this shows that the order of $(-\zeta_{n})^{-r'_{n}+ r'_{n\ell}}$ is prime to $q$.

As $q$ is an arbitrary divisor of $n$ it therefore follows that $(-\zeta_{n})^{-r'_{n} + r'_{n\ell}}$ is equal to $1$ if $n$ is even and to either $\pm 1$ if $n$ is odd, and hence that $\pi^{n\ell}_n(r'_{n\ell}) \equiv r'_{n}$ (mod $\mathcal{T}^\ast_{n}$), as required.  \end{proof}

\subsection{}\label{proof of torsion} We are now ready to prove Theorem \ref{reduction result}.

The final assertion of claim (i) follows directly from the final assertion of Lemma \ref{reduction to torsion lemma} and it is also clear that $\mathcal{F}^{\rm d}_{\rm tv}$ contains both $\mathcal{F}^{\rm d}_{\rm tor}$ and $(1-\tau)\mathcal{F}^{\rm c}$. To prove claim (i) it is thus enough to show that $\mathcal{F}^{\rm d}_{\rm tv}$ is contained in $\mathcal{F}^{\rm d}_{\rm tor}+ (1-\tau)\mathcal{F}^{\rm c}$.

To do this we take $f$ in $\mathcal{F}^{\rm d}_{\rm tv}$ and use the elements $(r_n')_n$ that arise in the equality (\ref{def eq}). Taking account firstly of Proposition \ref{inverse limit prop}, and then of the exact sequence in Lemma \ref{surjective lemma}(iii), we deduce the existence of an element $r$ of $R$ such that, for each $n \in \mathbb{N}^*$, the image of $r(1-\tau)$ in $R_n(1-\tau)$ is equal to $r_n'(1-\tau) + t_n(1-\tau)$ for an element $t_n$ of $\mathcal{T}_n^*$. This implies
\begin{equation}\label{conceptual} f(n) = (1-\zeta_n)^{r'_n(1-\tau)} = (1-\zeta_n)^{r(1-\tau)}\cdot(1-\zeta_n)^{t_n(1-\tau)} = \Phi^{(1-\tau)r}(n)\cdot (-\zeta_n)^{t_n}.\end{equation}

In particular, since the annihilator $\mathcal{T}_n$ in $R_n$ of $-\zeta_n$ is a submodule of $\mathcal{T}_n^*$ of index at most two, this shows that the distribution $f\cdot \Phi^{-(1-\tau)r}$ has order at most two and hence belongs to $\mathcal{F}^{\rm d}_{\rm tor}$, as suffices to complete the proof of claim (i).

Since $\mathcal{F}^{\rm c}$ is torsion-free the intersection $\mathcal{F}^{\rm d}_{\rm tor}\cap \mathcal{F}^{\rm c}$ is trivial and so the first isomorphism stated in claim (ii) follows directly from the equality
$\mathcal{F}^{\rm d}_{\rm tv} = \mathcal{F}^{\rm d}_{\rm tor}+ (1-\tau)\mathcal{F}^{\rm c}$ that is proved in claim (i).

Write $\zeta$ for the topological generator $(\zeta_m)_m$ of $\widehat{\ZZ}(1)$. Then, to complete the proof of claim (ii), we need to show that the assignment 
\begin{equation}\label{exp iso} \Phi^{(1-\tau)r} \mapsto \zeta^r\end{equation}
for each $r$ in $R$ gives a well-defined isomorphism of $R$-modules $(1-\tau)\mathcal{F}^c \cong \widehat{\ZZ}(1)$.

To check this we note first that for any $r$ and $r'$ in $R$ one has
\begin{align*} \Phi^{(1-\tau)r} = \Phi^{(1-\tau)r'} &\Longleftrightarrow \Phi^{(1-\tau)(r-r')} = 1 \\
 &\Longleftrightarrow r-r' \in R \cap \prod_{n}\mathcal{T}_n\\
 &\Longleftrightarrow r-r' \in \varprojlim_n\mathcal{T}^*_n\\
 &\Longleftrightarrow \zeta^{r-r'} = 1.\end{align*}
Here the first equivalence is clear, the second follows from the fact that $\Phi^{1-\tau}(n) = -\zeta_n$ for all $n$, the third from the fact that
$R \cap \prod_{n}\mathcal{T}_n = \varprojlim_n\mathcal{T}^*_n$ since $\mathcal{T}_n = \mathcal{T}^*_n$ for all even $n$ and the final equivalence is true because $\mathcal{T}^*_n$ is equal to the annihilator in $R_n$ of $\zeta_n$ whenever $n$ is a multiple of $4$.

This shows that the assignment (\ref{exp iso}) is both well-defined and injective. The proof of claim (ii)  is then completed by noting this map is
 clearly also both surjective and a homomorphism of $R$-modules.

Next we note that the surjectivity of the map given in claim (iii) is clear since the torsion subgroup of $\mathcal{F}^{\rm d}$ is annihilated by $1+\tau$.

To prove claim (iii) it is thus enough to show that if $f$ is any element of $\mathcal{F}^*$ such that $f^{1+\tau}$ belongs to $(1+\tau)\mathcal{F}^{\rm c}$, then $f$ must belong to the subgroup $\mathcal{F}^*_{\rm tor} + \mathcal{F}^{\rm c}$.

Now, for any such $f$ there exists an element $r$ of $R$ with $f^{1+\tau} = (\Phi^r)^{1+\tau}$ and so the product distribution $f_1 := f\cdot \Phi^{-r}$ satisfies the condition (\ref{key point}).

By the argument in claim (i) this implies $f_1$ belongs to $\mathcal{F}^*_{\rm tor} + (1-\tau)\mathcal{F}^{\rm c}$ and hence that
 $f = f_1\cdot \Phi^r$ belongs to $\mathcal{F}^*_{\rm tor} + (1-\tau)\mathcal{F}^{\rm c} + \mathcal{F}^{\rm c} = \mathcal{F}^*_{\rm tor} + \mathcal{F}^{\rm c}$, as required.

This completes the proof of Theorem \ref{reduction result}.

\begin{remark}\label{more concept}{\em The equality (\ref{conceptual}) and isomorphism (\ref{exp iso}) combine to give a more conceptual proof of the equality $\mathcal{F}^{\rm d}_{\rm tor} = \mathcal{D}$ that was discussed in the Introduction and first proved by the second author in \cite{Seo3}. Specifically, the equality (\ref{conceptual}) shows that for any $f$ in $\mathcal{F}_{\rm tor}^{\rm d}$ one has $f=f_1\cdot f_2$ with $f_1 \in (1-\tau)\mathcal{F}^{\rm c}$ and $f_2\in \mathcal{F}^{\rm d}$ such that $f_2(n) = (-\zeta_n)^{t_n}$ for all $n$, where each $t_n$ belongs to $\mathcal{T}_n^*$ and so $f_2(n)\in \langle (-1)^n\rangle$. Thus, since $(1-\tau)\mathcal{F}^{\rm c}$ is torsion-free, it follows that $f=f_2$ is such that $f(n)$ belongs to $\{\pm 1\}$ and can be non-trivial only if $n$ is odd. Given this, it is then easy to check directly from the distribution relations (\ref{basic facts}) that $f$ must be a Coleman distribution $\delta_\Pi$ for a suitable set of odd prime numbers $\Pi$.}\end{remark}

\section{Distributions of prime level}\label{pd section}

In this section we consider distributions of prime level, as defined in \S\ref{distributions of level}.

\subsection{}\label{prime level section} For brevity we set $\overline{\mathcal{F}}^{\rm d} := (1+\tau)\mathcal{F}^{\rm d}$ and $\overline{\mathcal{F}}^{{\rm c}} := (1+\tau)\mathcal{F}^{\rm c}$ and for each $m$ in $\mathbb{N}^*$ also $\overline{\mathcal{F}}^{\rm d}_{(m)} := (1+\tau)\mathcal{F}^{\rm d}_{(m)}$.

We recall that $\iota_m$ denotes the natural restriction map $\mathcal{F}^{\rm d} \to \mathcal{F}^{\rm d}_{(m)}$ and we set $\mathcal{F}^{\rm c}_{(m)} := \iota_m(\mathcal{F}^{\rm c})$ and $\overline{\mathcal{F}}_{(m)}^{{\rm c}} := (1+\tau)\mathcal{F}_{(m)}^{\rm c}$. For each prime $p$ we then write
\[ \kappa_p: \overline{\mathcal{F}}^{\rm d}/\overline{\mathcal{F}}^{{\rm c}} \to
\overline{\mathcal{F}}^{\rm d}_{(p)}/\overline{\mathcal{F}}_{(p)}^{{\rm c}}\]
for the homomorphism that is induced by $\iota_p$.

Finally we recall the torsion-free subgroups $V(n)$ of $\QQ(n)^{+,\times}$ that are defined in \S\ref{totally positive section}.

In terms of this notation, we can now state the main result that we prove concerning distributions of prime level.

\begin{theorem}\label{key real} For each prime $p$ the following claims are valid.
\begin{itemize}
\item[(i)]  The homomorphism $\kappa_p$ is injective and its cokernel is annihilated by $1- \sigma_p$.
\item[(ii)] The quotient group $\overline{\mathcal{F}}^{\rm d}_{(p)}/\overline{\mathcal{F}}^{{\rm c}}_{(p)}$ is uniquely $p$-divisible.
\item[(iii)] The following conditions are equivalent.
\begin{itemize}
\item[(a)] The quotient group $\mathcal{F}_{\rm tf}^{\rm d}/\mathcal{F}^{{\rm c}}$ is uniquely $p$-divisible.
\item[(b)] If $f$ is any distribution in $\mathcal{F}^{{\rm d},+}$ with the property that $f(n) \in V(n)^{p}$ for all $n \in \mathbb{N}(p)$, then $f(n) \in V(n)^p$ for all $n$.
\item[(c)] There exists a natural number $t$ such that if $f$ is any distribution in $\mathcal{F}^{{\rm d},+}$ with the property that $f(n) \in V(n)^{p^t}$ for all $n \in \mathbb{N}(p)$, then $f(n) \in V(n)^p$ for all $n$.
\end{itemize}

 \end{itemize}
\end{theorem}

In Remark \ref{p adic prop} below we will discuss the explicit condition that occurs in Theorem \ref{key real}(iii)(c) in the context of cyclotomic distributions.

The proof of this result will occupy the rest of \S\ref{pd section}. In the sequel we shall therefore fix a prime $p$.

\subsection{}In this section we prove Theorem \ref{key real}(i).

To verify that $\kappa_p$ is injective, we fix an $f$ in $\mathcal{F}^{\rm d}$ such that $\iota_p(f^{1+\tau})\in (1+\tau)\iota_p(\mathcal{F}^{\rm c})$ and, under these hypotheses, we must show that $f^{1+\tau}\in (1+\tau)\mathcal{F}^{{\rm c}}$.

It is therefore enough to show that if $r$ is any element of $R$ for which $\iota_p(f^{1+\tau}) = \iota_p(\Phi^{r(1+\tau)})$, then one has $f^{1+\tau}= \Phi^{r(1+\tau)}$.

But if we fix such an $r$ then the product distribution $\Delta:= f^{1+\tau}\cdot \Phi^{-r(1+\tau)}$ belongs to the module
$\mathcal{F}^{{\rm d},+}$ defined in \S\ref{totally positive section} and also, by assumption, to the kernel of $\iota_p$.

Hence, by applying Lemma \ref{useful 1} (with $\Sigma = \mathbb{N}(p)$), we can conclude that $\Delta$ is trivial, as required.

To show ${\rm cok}(\kappa_p)$ is annihilated by $1-\sigma_p$ it is enough to show that for any given function $f$ in $\mathcal{F}^{\rm d}_{(p)}$ there exists a function $f'$ in $\mathcal{F}^{\rm d}$ for which one has $f^{1-\sigma_p} = \iota_p(f')$.

To do this we define $f'$ to be the unique $\Gal(\QQ^c/\QQ)$-equivariant map $\mu_\infty^* \to \QQ^{c,\times}$ that satisfies
\[ f'(\zeta_n) := \begin{cases} f(n)^{1-\sigma_p}, &\text{ if $p$ divides $n$,}\\
                               {\rm N}^{np}_n(f(np)), &\text{ if $p$ is prime to $n$.}\end{cases}\]
It is then immediately clear that $\iota_p(f') = f^{1-\sigma_p}$ whilst an easy exercise shows that $f'$ inherits the properties (\ref{defining prop}) and (\ref{defining prop2}) from $f$ and hence that $f'$ belongs to $\mathcal{F}^{\rm d}$, as required.

This completes the proof of Theorem \ref{key real}(i).

\begin{remark}{\em The same argument can be used to show that for any natural numbers $m$ and $n$ the natural `restriction' map $\overline{\mathcal{F}}_{(m)}/\overline{\mathcal{F}}_{(m)}^{{\rm c}} \to \overline{\mathcal{F}}_{(mn)}/\overline{\mathcal{F}}^{{\rm c}}_{(mn)}$
is injective and has cokernel annihilated by $\prod_\ell(1-\sigma_\ell)$ where $\ell$ runs over all primes that divide $n$ but not $m$.} \end{remark}

\subsection{}In this section we prove Theorem \ref{key real}(ii).

\subsubsection{}For each $n$ in $\mathbb{N}^*$ we set $R_n^+ := \ZZ[G_n^+]$ and write $D(n)^+$ for the $R_n^+$-submodule of $V(n)$ that is generated by $\varepsilon_n := (1-\zeta_n)^{1+\tau}$.

We then define $\mathcal{V}_p^{\rm d}$ to be the subgroup of $\prod_{m \in \mathbb{N}(p)}D(m)^+_p$ comprising all elements $(x_m)_m$ with the property that for any $m$ in ${\mathbb N}(p)$ and any prime $\ell$ one has
\[ {\rm N}^{m\ell}_m(x_{m\ell}) = \begin{cases} x_m, &\text{ if $\ell$ divides $m$,}\\
                                                x_m^{1-\sigma_\ell}, &\text{ otherwise.}\end{cases}\]

Now if $f$ belongs to $\mathcal{F}_{(p)}^{\rm d}$, then for each $n$ in $\mathbb{N}(p)$ the value of $f^{1+\tau}$ at $\zeta_n$ belongs to the (torsion-free) group $V(n)$ and hence can be regarded as an element of $V(n)_p$.

In view of Theorem \ref{first step}, we can therefore use the assignment $f^{1+\tau} \leftrightarrow \{f(m)^{1+\tau}\}_m$ to identify  $\overline{\mathcal{F}}_{(p)}^{\rm d}$ as a subgroup of $\mathcal{V}_p^{\rm d}$.

We can then consider the exact commutative diagram

\[ \begin{CD} 0 @> >> \overline{\mathcal{F}}_{(p)}^{\rm d}/\overline{\mathcal{F}}_{(p)}^{{\rm c}} @> >> \mathcal{V}_{p}^{\rm d}/\overline{\mathcal{F}}_{(p)}^{{\rm c}} @> >> \mathcal{V}_{p}^{\rm d}/\overline{\mathcal{F}}_{(p)}^{{\rm d}} @> >> 0\\
@. @V \times p VV  @V \times p VV @V \times p VV\\
0 @> >> \overline{\mathcal{F}}_{(p)}^{\rm d}/\overline{\mathcal{F}}_{(p)}^{{\rm c}} @> >> \mathcal{V}_{p}^{\rm d}/\overline{\mathcal{F}}_{(p)}^{{\rm c}} @> >> \mathcal{V}_{p}^{\rm d}/\overline{\mathcal{F}}_{(p)}^{{\rm d}} @> >> 0\end{CD}\]
in which the rows are induced by the inclusions $\overline{\mathcal{F}}_{(p)}^{\rm c} \subseteq \overline{\mathcal{F}}_{(p)}^{\rm d}\subseteq \mathcal{V}_p^{\rm d}$.

From Proposition \ref{technical 2}(ii) and (iii) below we know that the second and third vertical arrows in this diagram are respectively bijective and injective and hence (by an application of the Snake Lemma) that the first vertical arrow is bijective.

This observation shows that the group $\overline{\mathcal{F}}_{(p)}^{\rm d}/\overline{\mathcal{F}}_{(p)}^{{\rm c}}$ is uniquely $p$-divisible, and hence proves Theorem \ref{key real}(ii).

\subsubsection{}We recall that $I_n$ denotes the annihilator in $R^+_n$ of the element $\varepsilon_n$ and that this ideal is explicitly described in Lemma \ref{useful 3}.

Before stating the next result we note that $\mathcal{V}_p^{\rm d}$ is naturally a module over the algebra
\[ \Lambda_{(p)} := \varprojlim_{m\in\mathbb{N}(p)} R_{m,p}^+,\]
where the limit is taken with respect to the natural projection maps $R_{n}^+ \to R^+_m$ for each $m$ in $\mathbb{N}(p)$ and each $n$ in $\mathbb{N}(m)$.

\begin{proposition}\label{technical 2} \
\begin{itemize}
\item[(i)] The $\Lambda_{(p)}$-module $\mathcal{V}^{\rm d}_p$ is free of rank one with basis $\,\iota_p(\Phi^{1+\tau})$.
\item[(ii)] The quotient group $\mathcal{V}^{\rm d}_p/\overline{\mathcal{F}}^{\rm c}_{(p)}$ is uniquely $p$-divisible.
\item[(iii)] The quotient group $\mathcal{V}^{\rm d}_p/\overline{\mathcal{F}}^{\rm d}_{(p)}$ has no element of order $p$.
\end{itemize}
\end{proposition}

\begin{proof} For each $n$ in $\mathbb{N}(p)$ the map $R_{n,p}^+ \to D(n)^+_p$ that sends $1$ to $\varepsilon_n$ induces an isomorphism of $R_{n,p}$-modules $R^+_{n,p}/I_{n,p} \cong D(n)^+_p$. We therefore obtain a well-defined isomorphism of $\Lambda_{(p)}$-modules
\[ \varepsilon: \prod_{n\in \mathbb{N}(p)}R^+_{n,p}/I_{n,p} \to \prod_{n\in \mathbb{N}(p)}D(n)^+_p\]
by defining $\varepsilon ((r_n)_n)$ to be $(\varepsilon_n^{r_n})_n$.

Then the norm relations for the cyclotomic elements $\varepsilon_n$ combine to imply that $\varepsilon$ restricts to give an isomorphism
\[ \mathcal{W}_p \cong \mathcal{V}^{\rm d}_p\]
where $\mathcal{W}_p$ denotes the $\Lambda_{(p)}$-submodule of $\prod_{n\in \mathbb{N}(p)}R^+_{n,p}/I_{n,p}$ comprising elements with the property that for every $m$ in $\mathbb{N}(p)$ and every prime $\ell$ one has
\begin{itemize}
\item[($*_1$)] $\pi^{m\ell}_m(r_{m\ell}) \equiv r_m$ modulo $I_{m,p}$ if $\ell$ divides $m$;
\item[($*_2$)] $\pi^{m\ell}_m((1-\sigma_\ell)\cdot r_{m\ell}) \equiv (1-\sigma_\ell)\cdot r_m$ modulo $I_{m,p}$ if $\ell$ is prime to $m$.
\end{itemize}
(Note that these congruence conditions are well-defined since for any $m$ and any prime $\ell$ one has $\pi^{m\ell}_m(I_{m\ell}) \subseteq I_{m}$ if $\ell$ divides $m$ and $(1-\sigma_\ell)\cdot \pi^{m\ell}_m(I_{m\ell}) \subseteq I_{m}$ if $\ell$ is prime to $m$.)

We write
\[ \Delta_p: \Lambda_{(p)}\to \mathcal{W}_p\]
for the natural diagonal map.

Then for each $r$ in $\Lambda_{(p)}$ and each $n$ in $\mathbb{N}(p)$ one has $\varepsilon(\Delta_p(r))_n = \varepsilon_n^r = \iota_p(\Phi^{(1+\tau)})^r(n)$ and so $\varepsilon(\Delta_p(r)) = \iota_p(\Phi^{(1+\tau)})^r$. To prove claim (i) it is thus enough to show that $\Delta_p$ is bijective.

Injectivity of $\Delta_p$ follows directly from the first claim in Lemma \ref{derived proj limit lemma} below. It is thus enough to show that each element $r = (r_m)_{m \in \mathbb{N}(p)}$ of $\mathcal{W}_p$ belongs to the image of $\Delta_p$.

Now, for each such $r$ and each $m$ in $\mathbb{N}(p)$, the condition ($*_1$) implies that the element $r_{mp^\infty} := (r_{mp^d})_{d \ge 0}$ belongs to $\varprojlim_d R^+_{mp^d,p}/I_{mp^d,p}$, where the transition morphisms are the natural projection maps, whilst Lemma \ref{derived proj limit lemma} below implies that the natural projection map from $\Lambda_{(m,p)} := \varprojlim_d R^+_{mp^d,p}$ to $\varprojlim_d R^+_{mp^d,p}/I_{mp^d,p}$ is bijective.

For each $m$ in $\mathbb{N}(p)$ and each multiple $m'$ of $m$ we therefore obtain a surjective composite homomorphism
\[ \varpi^{m'}_m: \varprojlim_d R^+_{m'p^d,p}/I_{m'p^d,p} \cong \Lambda_{(m',p)} \to \Lambda_{(m,p)} \cong \varprojlim_d R^+_{mp^d,p}/I_{mp^d,p}\]
in which the arrow denotes the natural projection map.

We now assume that $m' = m\ell$ for a prime $\ell$. If $\ell$ divides $m$, then ($*_1$) implies directly that $\varpi^{m\ell}_m(r_{m\ell p^\infty}) = r_{mp^\infty}$. In addition, if $\ell$ does not divide $m$, then ($*_2$) implies that the difference $\varpi^{m\ell}_m(r_{m\ell p^\infty}) - r_{mp^\infty}$ is annihilated by multiplication by $1-\sigma_\ell$ and hence, since $1-\sigma_\ell$ is a non-zero divisor in $\Lambda_{(m,p)}$, that $\varpi^{m\ell}_m(r_{m\ell p^\infty}) = r_{mp^\infty}$.

Since $\Lambda_{(p)}$ is equal to $\varprojlim_{m}\Lambda_{(m,p)}$, where $m$ runs over $\mathbb{N}(p)$ and the transitions morphisms are the natural projection maps, this shows that $r$ belongs to the image of $\Delta_p$, as required to complete the proof of claim (i).

Turning to claim (ii) we note that the subgroup $\overline{\mathcal{F}}_{(p)}^{\rm c}$ of $\mathcal{V}^{\rm d}_p$ coincides with the image of the composite homomorphism
\begin{equation*}\label{composite} R \to \varprojlim_{m\in \mathbb{N}(p)}R_m^+ \xrightarrow{\theta} \varprojlim_{m\in \mathbb{N}(p)}R_{m,p}^+ = \Lambda_{(p)} \xrightarrow{\Delta_p} \mathcal{W}_p \xrightarrow{\varepsilon} \mathcal{V}^{\rm d}_p.\end{equation*}
where the first map is the natural projection and $\theta$ is the natural inclusion.

In particular, since the first map in this composition is surjective and the last two are bijective, claim (ii) will follow if we can show that  ${\rm cok}(\theta)$  is uniquely $p$-divisible.

To do this we note that for each $m$ in $\mathbb{N}(p)$ and each $m'$ in $\mathbb{N}(m)$ there exists a commutative diagram of short exact sequences

\[\begin{CD} 0 @> >> R_{m'}^+ @> \subseteq >> R^+_{m',p} @> >> (\ZZ_p/\ZZ)\otimes_\ZZ R_{m'}^+ @> >> 0\\
 @. @V VV @V VV @V VV \\
 0 @> >> R_{m}^+ @> \subseteq >> R^+_{m,p} @> >> (\ZZ_p/\ZZ)\otimes_\ZZ R_m^+ @> >> 0\end{CD}\]
in which each vertical arrow is the natural projection map, and so is surjective.

Hence, by passing to the inverse limit over $m$ of these sequences, we deduce that ${\rm cok}(\theta)$ is isomorphic to $\varprojlim_m (\ZZ_p/\ZZ)\otimes_\ZZ R_{m}^+$ and hence is uniquely $p$-divisible since $\ZZ_p/\ZZ$ is.

To prove claim (iii) we note that $\mathcal{V}^{\rm d}_p/\overline{\mathcal{F}}^{\rm d}_{(p)}$ is a quotient of the uniquely $p$-divisible group $\mathcal{V}^{\rm d}_p/\overline{\mathcal{F}}^{\rm c}_{(p)}$ and so has no element of order $p$ if and only if the subgroup of elements of $p$-power order has bounded exponent.

To investigate this condition we write $\mathcal{U}$ for the subgroup of $\mathcal{F}^{\rm d}_{(p)}$ comprising all functions $f$ with the property that for every $n$ in $\mathbb{N}(p)$ the value $f(n)$ belongs to the (torsion-free) subgroup $V(n)$ of $E(n)'$.

It is then clear that $\overline{\mathcal{F}}^{\rm d}_{(p)}\subseteq \mathcal{U}$ and that an element $x = (x_n)_n$ of $\mathcal{V}^{\rm d}_p$ is equal to $(f(n))_n$ for some $f$ in $\mathcal{U}$ if and only if one has $x_n\in V(n)$ for every $n$.

We now assume to be given an element $c$ of $\mathcal{V}^{\rm d}_p/\overline{\mathcal{F}}^{\rm d}_{(p)}$ of exact order $p^t$ for some $t$ and we lift $c$ to an element $x = (x_n)_n$ of $\mathcal{V}^{\rm d}_p$.

Then, for each $n \in \mathbb{N}(p)$, we know that the element $x_n^{p^t}$ of $D(n)^+_p \subseteq V(n)_p$ belongs to $V(n)$. Since the quotient of $V(n)_p$ by $V(n)$ is uniquely $p$-divisible, it follows that $x_n$ belongs to $V(n)$ and hence that $x = (f(n))_n$ for some $f$ in $\mathcal{U}$.

Since $f^2 = f^{1+\tau}$ belongs to $\overline{\mathcal{F}}^{\rm d}_{(p)}$ this in turn implies that the order $p^t$ of $c$ divides $2$, as suffices to prove claim (iii). \end{proof}

We end this section by proving a technical result about limits that was used above.

\begin{lemma}\label{derived proj limit lemma} Fix $m$ in $\mathbb{N}^\ast$ and a prime divisor $p$ of $m$. For each natural number $b$ write $\pi'_b$ for the restriction of $\pi^{mp^{b+1}}_{mp^{b}}$ to $I_{mp^{b+1}}$ and $\pi'_{b,p}$ for the scalar extension $\ZZ_p\otimes_\ZZ\pi'_b$.

Then the limit ${\varprojlim}_b (I_{mp^b,p},\pi'_{b,p})$ and first derived limit ${\varprojlim}^1_b (I_{mp^b,p},\pi'_{b,p})$ both vanish.
\end{lemma}

\begin{proof} The vanishing of the derived limit ${\varprojlim}^1_b (I_{mp^b,p},\pi'_{b,p})$ follows directly from the Mittag-Leffler criterion and the fact that each module $I_{mp^b,p}$ is compact.

We note next that there exists a natural number $b_0$ such that all prime divisors of $m$ have full decomposition group in $\Gal(\QQ(mp^\infty)^+/\QQ(mp^{b_0})^+)$.

Then to prove the vanishing of ${\varprojlim}_b (I_{mp^b,p},\pi'_{b,p})$ it is clearly enough to prove that for all $b \ge b_0$ one has
\begin{equation}\label{ud claim} \im(\pi_b') = p\cdot I_{mp^{b}}\end{equation}
and hence also $\im(\pi_b')_p = p\cdot I_{mp^{b},p}$. To prove this we fix $b \ge b_0$ and set $M := mp^b$, $M_0 := mp^{b_0}$, $G := G_M^+$ and $G_0 := G_{M_0}^+$.

If $M$ is a power of $p$, then (\ref{ud claim}) follows immediately from the fact that $I_{M}$ vanishes for all values of $b$ by Example \ref{annihilator exam}(i).

We therefore assume $M$ is divisible by at least two primes. In this case the description of $I_M$ given in Lemma \ref{useful 3} implies that, for any non-trivial homomorphism $\psi: G \to \QQ^{c,\times}$, the idempotent $e_\psi$ belongs to $\QQ^c\otimes_\ZZ I_M$  if and only if there exists a prime $\ell$ that divides $m$ and is such that its decomposition subgroup in $G$ is contained in $\ker(\psi)$. This implies, in particular, that any such $\psi$ factors through the projection $G \to G_0$.

Now any element $x$ of $I_{M}$ can be written uniquely as $\sum_{\psi}c_\psi\cdot e_\psi$ for suitable elements $c_\psi$ of $\QQ^c$, where $\psi$ runs over all homomorphisms $G \to \QQ^{c,\times}$.

In particular, if any term $c_\psi$ in this sum is non-zero, then (since $\varepsilon_M$ generates an $R^+_M$-module that is torsion-free) $e_\psi$ must belong to $\QQ^c\otimes_\ZZ I_M$ and hence $\psi$ factors through the projection $G \to G_0$.

This observation implies that $x = hx$ for all elements $h$ of $H := \Gal(\QQ(M)^+/\QQ(M_0)^+)$ and hence that $I_{M}$ is equal to the set of elements of the form $x'\cdot \sum_{h \in H}h$ with $x' \in R_M^+$ such that $\pi^{M}_{M_0}(x')\in I_{M_0}$.

By using this description, one shows that the restriction of $\pi^M_{M_0}$ to $I_{M}$ is both injective and has image $|H|\cdot I_{M_0}$ and then (\ref{ud claim}) follows easily from this fact.
\end{proof}

\begin{remark}{\em A closer analysis of the above argument shows that the inverse system of abelian groups $(I_{mp^{b}},\pi_b')_{b \in \mathbb{N}}$ is isomorphic to a direct sum of finitely many copies of the system $(X_b,\kappa_b)_{b \in \mathbb{N}}$ where $X_b := \ZZ$ for each $b$ and each transition morphism $\kappa_b$ is multiplication by $p$ and this fact can be used to show that the derived limit ${\varprojlim}^1_b (I_{mp^b},\pi'_b)$ is uniquely $p$-divisible. However, we make no use of this fact and so omit the details. }\end{remark}

\begin{remark}\label{p adic prop}{\em Proposition \ref{technical 2}(i) also allows us to show that cyclotomic distributions satisfy the condition in Theorem \ref{key real}(iii)(c) with $t = 1$ if $p$ is odd and $t = 2$ if $p=2$. To explain this, we fix $r$ in $R$ such that the distribution $f := \Phi^{(1+\tau)r}$ has $f(m) \in V(m)^{p^t}$ for all $m$ in $\mathbb{N}(p)$, with $t$ as specified above. Then for any $n$ in $\mathbb{N}^*\setminus \mathbb{N}(p)$ the sequence $(f(np^a))_{a \in \mathbb{N}}$ is a $p^t$-th power in the limit $\mathcal{F}^{\rm u}(np)_p^\infty$ that occurs in Proposition \ref{bijection lemma} and hence also in the limit $C(np)_p^\infty$. The argument of \cite[Th. 2.4]{Seo3} then allows one to deduce that each element $f(np^a)^{2} = f(np^a)^{1+\tau}$ is a $p^t$-th power in the $R^+_{np^a,p}$-module generated by $\Phi^{1+\tau}(np^a)$. This then combines with the result of Proposition \ref{technical 2}(i) to imply that the element of $\Lambda_{(p)}$ corresponding to $(1+\tau)2r$ is divisible by $p^t$, and hence in all cases that $(1+\tau)r$ belongs to $pR$. This in turn guarantees that $f(n)$ belongs to $V(n)^p$ for all $n$ in $\mathbb{N}^*$, as required.
}\end{remark}

\subsection{}We now prove Theorem \ref{key real}(iii).

To do this we first note that Theorem \ref{reduction result}(iii) implies $\mathcal{F}_{\rm tf}^{\rm d}/\mathcal{F}^{{\rm c}}$ is uniquely $p$-divisible if and only if $\overline{\mathcal{F}}^{\rm d}/\overline{\mathcal{F}}^{\rm c}$ is uniquely $p$-divisible. Then we use the fact that Theorem \ref{key real}(i) gives rise to an exact commutative diagram

\[ \begin{CD} 0 @> >> \overline{\mathcal{F}}^{\rm d}/\overline{\mathcal{F}}^{{\rm c}} @> \kappa_p >> \overline{\mathcal{F}}_{(p)}^{\rm d}/\overline{\mathcal{F}}_{(p)}^{{\rm c}} @> >> {\rm cok}(\kappa_p) @> >> 0\\
@. @V \times p VV  @V \times p VV @V \times p VV\\
0 @> >> \overline{\mathcal{F}}^{\rm d}/\overline{\mathcal{F}}^{{\rm c}} @> \kappa_p >> \overline{\mathcal{F}}_{(p)}^{\rm d}/\overline{\mathcal{F}}_{(p)}^{{\rm c}} @> >> {\rm cok}(\kappa_p) @> >> 0\end{CD}\]
in which the central vertical arrow is bijective by Theorem \ref{key real}(ii).

In particular, by applying the Snake Lemma to this diagram we deduce that $\overline{\mathcal{F}}^{\rm d}/\overline{\mathcal{F}}^{{\rm c}}$ is uniquely $p$-divisible if and only if the $p$-power torsion subgroup ${\rm cok}(\kappa_p)[p^\infty]$ of ${\rm cok}(\kappa_p)$ vanishes. Further, since $\overline{\mathcal{F}}_{(p)}^{\rm d}/\overline{\mathcal{F}}_{(p)}^{{\rm c}}$ is $p$-divisible, the group ${\rm cok}(\kappa_p)[p^\infty]$ is also $p$-divisible and so vanishes if and only if it has bounded exponent.

To prove equivalence of the stated conditions in Theorem \ref{key real}(iii) it is thus enough to show that condition (c) implies ${\rm cok}(\kappa_p)[p^\infty]$ has bounded exponent.

To do this we assume the validity of condition (c), fix an element $x$ of ${\rm cok}(\kappa_p)$ of exact order $p^e$ and can further assume, without loss of generality, that $e$ is strictly bigger than the integer $t$ that occurs in condition (c).

We choose a map $f$ in $\mathcal{F}^{\rm d}_{(p)}$ such that $f^{1+\tau}$ represents $x$ and note that, by assumption, there exists a map $f_e$ in $\mathcal{F}^{\rm d}$ with $\iota_p(f_e^{1+\tau}) = f^{(1+\tau)p^e}$.

By applying condition (c) to the map $f_e^{1+\tau}$ we can thus deduce that for every $n$ in $\mathbb{N}^*$ there exists a (unique) element $h_{1,n}$ of the (torsion-free) group $V(n)$ with $f_e(n)^{1+\tau} = h_{1,n}^p$.

%
%
%
%

%

We then define $h_1: \mu_\infty^*\to \QQ^{c,\times}$ to be the unique $\Gal(\QQ^c/\QQ)$-equivariant map with $h_1(\zeta_n) = h_{1,n}$ for all $n$ in $\mathbb{N}^*$.

Then $h_1^{p} = f_e^{1+\tau}$ and hence, as each group $V(n)$ is torsion-free, the function $h_1$ inherits the relation (\ref{defining prop}) from the map $f_e^{1+\tau}$ and so belongs to $\mathcal{F}^{{\rm d},+}$.

In a similar way, one has
\[ \iota_p(h_1)^p = \iota_p(h_1^p) = \iota_p(f_e^{1+\tau}) =  f^{(1+\tau)p^e}\]
and hence $\iota_p(h_1) = f^{(1+\tau)p^{e-1}}$.

Since, by assumption, $e-1 \ge t$ we can now repeat the above argument with $f_e^{1+\tau}$ replaced by $h_1$ in order to deduce the existence of a function $h$ in $\mathcal{F}^{{\rm d},+}$ with $\iota_p(h) = f^{(1+\tau)p^{e-2}}$.


Then, as $h$ is fixed by $\tau$, the latter equality implies that $\iota_p(h^{1+\tau}) = \iota_p(h^2) = f^{(1+\tau)2p^{e-2}}$ and hence that the order of $x$ divides $2p^{e-2}$.

This contradicts the assumption that the order of $x$ is $p^e$, showing that the exponent of ${\rm cok}(\kappa_p)[p^\infty]$ must be bounded and hence  completing the proof of Theorem \ref{key real}(iii).

\section{The proof of Theorem \ref{main result}}\label{proofs section}

In this section we prove Theorem \ref{main result} and then also justify the comments that follow the statement of this result in the Introduction.

\subsection{}\label{proof section}To prove Theorem \ref{main result} we note at the outset that the results of Theorem \ref{reduction result}(i) and (ii) combine to give a canonical exact commutative diagram of $R$-modules
\begin{equation*}\label{prelim main}
\xymatrix{ \widehat{\ZZ}(1) \ar@{^{(}->}[r] & \calF_{\rm tf}^{\rm d} \ar@{->>}^{\mu}[r] & \overline{\mathcal{F}}^{\rm d}\\
 \widehat{\ZZ}(1)\ar@{=}[u] \ar@{^{(}->}[r] & \calF^{\rm c} \ar@{->>}^{\mu}[r] \ar@{^{(}->}[u] & \overline{\mathcal{F}}^{{\rm c}} \ar@{^{(}->}[u]}
\end{equation*}
in which $\mu$ sends each $f$ to $f^{1+\tau}$, the vertical arrows are the natural inclusions and, as in \S\ref{pd section}, we set $\overline{\mathcal{F}}^{\rm d} := (1+\tau)\mathcal{F}^{\rm d}$ and $\overline{\mathcal{F}}^{{\rm c}} := (1+\tau)\mathcal{F}^{\rm c}$.

To derive from this diagram the existence of a diagram as in Theorem \ref{main result} it is thus enough to describe a canonical injective homomorphism of $R$-modules
\begin{equation}\label{kappa def} \kappa': \overline{\mathcal{F}}^{\rm d}\to \widehat{R}(1+\tau)\end{equation}
with the property that $\kappa'(\overline{\mathcal{F}}^{\rm c})$ is equal to the image of the diagonal embedding of $R(1+\tau)$ in $\widehat{R}(1+\tau)$ (and one can then define the map $\kappa$ in Theorem \ref{main result} to be the composite $\kappa'\circ \mu$).

To do this we use the canonical direct product decomposition of rings
\[ \widehat{R} = \prod_p \widehat{R}^{p}.\]
Here the product is over all primes $p$ and $\widehat{R}^{p}$ denotes the pro-$p$ completion $\varprojlim_nR_{n,p}$ of $R$ (so that in the inverse limit $n$ runs over $\mathbb{N}$ and the transition maps are the standard projection maps).

Next we recall that the approach of \S\ref{pd section} implies (via Proposition \ref{technical 2}(i)) that for every $f$ in $\overline{\mathcal{F}}^{\rm d}$ there exists a unique element $r_{f,p}$ of $\widehat{R}^{p}(1+\tau)$ with the property that for each $n$ in $\mathbb{N}(p)$ one has $f(n) = \Phi(n)^{r_{f,p}}$ in $V(n)_p$.

We can therefore define a map as in (\ref{kappa def}) by specifying that for each $f$ in $\overline{\mathcal{F}}^{\rm d}$ one has
$\kappa'(f) = (r_{f,p})_p$.

With this definition, it is clear that $\kappa'$ is a homomorphism of $R$-modules, that $\kappa'$ is injective and also, since $\overline{\mathcal{F}}^{\rm c} = R(1+\tau)\cdot \Phi$, that $\kappa'(\overline{\mathcal{F}}^{\rm c}) = R(1+\tau)$.

This completes the proof Theorem \ref{main result}.

\begin{remark}\label{explicit kappa}{\em A closer analysis of the above construction gives an explicit description of the map $\kappa = \kappa'\circ\mu$ that occurs in Theorem \ref{main result}. To explain this we fix $f$ in $\mathcal{F}^{\rm d}$ and $m$ in $\mathbb{N}^*$. Then for every prime $p$ the argument in \S\ref{pd section} shows that for each $n$ in $\mathbb{N}$ there exists an element $a_n = a_{f,m,n,p}$ of $R^+_{mp^n,p}$ with $f^{1+\tau}(mp^n) = \varepsilon_{mp^n}^{a_n}$ in $V(mp^n)_p$, that the sequence $(\pi^{mp^n}_m(a_n))_n$ is convergent in $R^+_{m,p}$, that the limit $a_{f,m,p} := \lim_{n\to \infty}\pi^{mp^n}_m(a_n)$ is independent of the choice of elements $a_n$ and that the tuple $(a_{f,m,p})_m$ belongs to $\varprojlim_m R^+_{m,p}$. Writing $a_{f,p}$ for the pre-image of $(a_{f,m,p})_m$ under the isomorphism of $\widehat R^p$-modules $\widehat R^p(1+\tau) \to \varprojlim_m R^+_{m,p}$ that sends $1+\tau$ to $1$, one then has $\kappa(f) = (a_{f,p})_p$.}\end{remark}

\begin{remark}{\em The exact sequence of $R$-modules given by the lower row of the diagram in Theorem \ref{main result} splits after inverting $2$ since the map sending $1+\tau$ to $\Phi^{(1+\tau)/2}$ is a section to the restriction of $\ZZ[\frac{1}{2}]\otimes_\ZZ\kappa$ to $\ZZ[\frac{1}{2}]\otimes_\ZZ\mathcal{F}^{\rm c}$ in this case. However, the sequence does not itself split, even as a sequence of $\Gal(\CC/\RR)$-modules, since if this was true there would exist $f$ in $\mathcal{F}^{{\rm c}}$ with $f = f^\tau$ and such that $\kappa(f) = 1+\tau$. This would in turn imply the existence of an element $r$ of $R$ with both $\Phi^{(1-\tau)r} = 1$ and $(1+\tau)r = 1+\tau$. The second equality here implies $r = 1-r'(1-\tau)$ for some $r'$ in $R$ and then the first equality implies $\Phi^{1-\tau} = \Phi^{2r'(1-\tau)}$.  But this cannot be true since, for example, $\Phi^{1-\tau}(4) = -\zeta_4$ is not a square in $\QQ(4)$.}\end{remark}

\begin{remark}\label{tate coh rem}{\em In this remark we set $\Gamma := \Gal(\CC/\RR)$ and note that for any $\Gamma$-module $M$ the Tate cohomology group $\hat H^0(\Gamma,M) := M^{\tau = 1}/(1+\tau)M$ is a vector space over the field of two elements $\mathbb{F}_2$. Then, since the module of Coleman distributions $\mathcal{D}$ is the torsion subgroup of $\mathcal{F}^{\rm d}$ (cf. Remark \ref{more concept}), there is a natural isomorphism %
%
\[ \hat H^0(\Gamma,\mathcal{F}_{\rm tf}^{\rm d}) \cong \frac{\{f \in \mathcal{F}^{\rm d}: f^{\tau-1}\in \mathcal{D}\}}{\mathcal{D} + (1+\tau)\mathcal{F}^{\rm d}}\]
and hence also an exact sequence of $\mathbb{F}_2$-vector spaces
\begin{equation}\label{first tate} 0 \to \mathcal{D} \to \hat H^0(\Gamma,\mathcal{F}^{\rm d}) \to \hat H^0(\Gamma,\mathcal{F}_{\rm tf}^{\rm d}) \to \mathcal{D}\cap (1-\tau)\mathcal{F}^{\rm d} \to 0.\end{equation}
Here the second map is induced by the inclusion $\mathcal{D} \subset \mathcal{F}^{{\rm d},\tau =1}$ (and so is injective since $\mathcal{D}\cap (1+\tau)\mathcal{F}^{{\rm d}} = 0$), the third is obvious and the fourth is induced by sending each distribution $f$ to $f^{1-\tau}$. In particular, since $\mathcal{D}$ is a vector space over $\mathbb{F}_2$ of uncountably infinite dimension, this shows that the group $\hat H^0(\Gamma,\mathcal{F}^{\rm d})$ is large. To be more explicit we assume Coleman's Conjecture to be valid so that $\mathcal{F}^{\rm d} = \mathcal{D} + \mathcal{F}^{\rm c}$. Then $\mathcal{D}\cap (1-\tau)\mathcal{F}^{\rm d}$ is equal to $\mathcal{D}\cap (1-\tau)\mathcal{F}^{\rm c}$ and so vanishes, and $\hat H^0(\Gamma,\mathcal{F}_{\rm tf}^{\rm d})$ identifies with $\hat H^0(\Gamma,\mathcal{F}^{\rm c})$. In addition, by considering Tate cohomology of the (exact) lower row of the diagram in Theorem \ref{main result}, one finds that the results of Theorem \ref{reduction result}(i) and (ii) combine to give a natural exact sequence
\[0 \to \hat H^0(\Gamma,\mathcal{F}^{\rm c}) \to \mathbb{F}_2\otimes_\ZZ R(1+\tau) \to \mathbb{F}_2 \to 0.\]
This sequence then combines with (\ref{first tate}) to give an exact sequence of the form
\[ 0 \to \mathcal{D} \to \hat H^0(\Gamma,\mathcal{F}^{\rm d}) \to \mathbb{F}_2\otimes_\ZZ R(1+\tau) \to \mathbb{F}_2 \to 0.\]
}\end{remark}

\begin{remark}\label{can intersect}{\em For each prime $p$ write $\mathcal{F}^{\dagger}_{(p)}$ for the $R$-submodule of $\mathcal{F}$ comprising functions whose image under the restriction map $\iota_p$ belongs to $\mathcal{F}^{\rm d}_{(p)}$. Then $\mathcal{F}^{\rm d}$ is the intersection of $\mathcal{F}^{\dagger}_{(p)}$ over all $p$ and so there is an inclusion of $\overline{\mathcal{F}}^{\rm d}= (1+\tau)\mathcal{F}^{\rm d}$ into $\bigcap_p (1+\tau)(\mathcal{F}^{\dagger}_{(p)})$ whose cokernel is annihilated by $2$. Taking account of the isomorphism in Theorem \ref{reduction result}(iii) one therefore obtains inclusions

%
\[ 2\cdot \bigcap_p \frac{(1+\tau)(\mathcal{F}^{\dagger}_{(p)})}{\overline{\mathcal{F}}^{\rm c}} \subseteq \frac{\mathcal{F}_{\rm tf}^{\rm d}}{\mathcal{F}^{\rm c}} \subseteq \bigcap_p \frac{(1+\tau)(\mathcal{F}^{\dagger}_{(p)})}{\overline{\mathcal{F}}^{\rm c}}\]
in which both intersections run over all $p$. For each $p$ there is also a natural exact sequence
\[ 0\to (1+\tau)\mathcal{F}\cap \ker(\iota_p) \to \frac{(1+\tau)(\mathcal{F}^{\dagger}_{(p)})}{\overline{\mathcal{F}}^{\rm c}} \to \frac{\overline{\mathcal{F}}^{\rm d}_{(p)}}{\overline{\mathcal{F}}^{\rm c}_{(p)}}\to 0,\]
in which the third map is induced by $\iota_p$ and exactness follows from the fact $\overline{\mathcal{F}}^{\rm c}$ is disjoint from $ \ker(\iota_p)$. In particular, since Lemma \ref{useful 1} implies $\overline{\mathcal{F}}^{\rm d}$ is also disjoint from $\ker(\iota_p)$, this sequence induces an identification of $\mathcal{F}_{\rm tf}^{\rm d}/\mathcal{F}^{\rm c}$ with a submodule of  $\overline{\mathcal{F}}^{\rm d}_{(p)}/\overline{\mathcal{F}}^{\rm c}_{(p)}$. Finally we recall that, by Theorem \ref{key real}(ii), the latter group is uniquely $p$-divisible.}\end{remark}

\begin{remark}{\em If Coleman's Conjecture is valid, and $f$ in $\mathcal{F}^{\rm d}$ is not torsion-valued, then the lower row of the diagram in Theorem \ref{main result} implies that the set of primes $\{\ell: \kappa(f) \in \ell\cdot \widehat{R}(1+\tau)\}$ is finite. The following result provides evidence that this is a reasonable expectation. }\end{remark}

\begin{lemma}\label{generic} If $f$ in $\mathcal{F}^{\rm d}$ is not torsion-valued, then the set of primes $\{\ell: \kappa(f) \in \ell\cdot \widehat{R}(1+\tau)\}$ has Dirichlet density zero. \end{lemma}

\begin{proof} Set $f' := f^{1+\tau}$. Then, since $f$ is not torsion-valued, Lemma \ref{reduction to torsion lemma} implies there exists $m$ in $\mathbb{N}^*$ such that $f'(m) \not= 1$. By using Lemma \ref{useful 1} we can also assume that $m$ is divisible by an odd prime $p$.

Then there exists an $n_0$ in $\mathbb{N}$ such that for all $n > n_0$ one has
$f'(mp^n)\notin \QQ(mp^{n-1})^+$. Indeed, if this is not true, then the norm compatibility of the elements $\{f'(mp^a)\}_{a \ge 0}$ implies  $f'(m)\in V(m)^{p^t}$ for arbitrarily large integers $t$ and, since $V(m)$ is torsion-free, this contradicts the assumption that $f'(m)\not= 1$.

Fix $n > n_0$ and an element $\sigma$ of $G_{mp^{n}}^+$ and write $\mathcal{P}_\sigma$ for the set of primes $q$ that do not divide $m$ and are such that the restriction of $\sigma_q$ to $\QQ(mp^{n})^+$ is equal to $\sigma$.

Then for any $q$ in the intersection of $\mathcal{P}_\sigma$ with $\mathcal{P}_{f} := \{\ell: \kappa(f) \in \ell\cdot \widehat{R}(1+\tau)\}$, the  distribution relation
\[ f'(mp^n)^{1-\sigma} = f'(mp^n)^{1-\sigma_q} = {\rm N}^{mp^nq}_{mp^n}(f'(mp^nq)) = {\rm N}^{mp^nq}_{mp^n}((\varepsilon_{mp^nq})^{\kappa(f)_q})\]
in $V(mp^n)_q$ implies that $f'(mp^n)^{1-\sigma}$ belongs to $V(mp^n)^q$.

Thus, if $\sigma$ is such that $\mathcal{P}_\sigma\cap \mathcal{P}_{f}$ is infinite, then $f'(mp^n)^{1-\sigma} = 1$. In any such case, the fact that $f'(mp^n)\notin \QQ(mp^{n-1})^+$ therefore implies that the subgroup generated by $\sigma$ must be disjoint from the cyclic subgroup $\Gal(\QQ(mp^{n})^+/\QQ(m'p)^+)$ of $G_{mp^n}^+$, where we write $m'$ for the maximal divisor of $m$ that is prime to $p$.

In particular, if the $p$-exponent of $G_{m'p}^+$ is $p^{n_1}$, then for any $n > n_0 + n_1$ the maximal power of $p$ that divides the order of $\sigma$ is at most $p^{n_1}$ whilst the maximal power of $p$ that divides the exponent of $G_{mp^n}^+$ is at least $p^{n-1}$.

This implies that $\#\{\sigma \in G_{mp^n}^+: \mathcal{P}_\sigma \cap \mathcal{P}_f \,\text{ is infinite }\}$
is at most $\#G_{mp^n}^+/p^{n-1-n_1}$ and hence, by the Tchebotarev Density Theorem, that the density of $\mathcal{P}_f$ is at most $p^{n_1+1-n}$.

%
%

The claimed result therefore follows from the fact that $p^{n_1+1-n}$ tends to $0$ as $n$ tends to infinity. \end{proof}

\subsection{}In this section we derive from Theorem \ref{main result} an explicit criterion for a distribution to be cyclotomic. This criterion is reminiscent of the archimedean `boundedness' criterion that Coleman uses in \cite{coleman2} to characterise cyclotomic units in $p$-power conductor abelian fields.

To do this we recall that, by the argument in \S\ref{pd section}, for each function $f$ in $\mathcal{F}^{\rm d}$ and each $m$ and $n$ in $\mathbb{N}$ there exists a (non-unique) element $a_{m,n}(f)$ in $R^+_{mp^n,p}$ such that in $V(mp^n)_p$ one has
\begin{equation}\label{key eq 2} f^{1+\tau}(mp^n) = (\varepsilon_{mp^n})^{a_{m,n}(f)}.\end{equation}

For any natural number $k$ with $k \le n$ and any element $x$ of $R^+_{mp^n,p}$ we write $|x|_k$ for the unique integer in $\{ a: 0\le a < p^{n-k}\}$ that is equal to the coefficient of the trivial element of $G_m^+$ in the standard representation of the image of $x$ under the composite map
\[ R^+_{mp^n,p} \to R^+_{m,p} \to \ZZ_p/(p^{n-k})\otimes _{\ZZ_p} R^+_{m,p} = \ZZ/(p^{n-k})\otimes_{\ZZ} R^+_{m},\]
where the first map is induced by $\pi^{mp^n}_m$ and the second is the natural reduction map.

We shall then say that a given subset of $\mathcal{F}^{\rm d}$ is `$p$-bounded' if for each of its elements $f$, each $m$ in $\mathbb{N}(p)$ and any sufficiently large $k$ at least one of the sets $\{ |a_{m,n}(f)|_k\}_{n> k}$ and $\{ |-\!a_{m,n}(f)|_k\}_{n > k}$ is bounded. (See Remark \ref{p-b indep} below).

\begin{theorem}\label{new criterion} Fix $f$ in $\mathcal{F}^{\rm d}$. Then $f$ is cyclotomic if and only if for any, and therefore for every, prime $p$ the set of distributions $\{f^\sigma: \sigma \in \Gal(\QQ^c/\QQ)\}$ is $p$-bounded.
\end{theorem}

\begin{proof} It is enough to fix a prime $p$ and show that $f$ is cyclotomic if and only if the set $\{f^\sigma: \sigma \in \Gal(\QQ^c/\QQ)\}$ is $p$-bounded.

At the outset, we note that the diagram in Theorem \ref{main result} implies directly that $f$ is cyclotomic if and only if $f^{1+\tau}$ belongs to $(1+\tau)\mathcal{F}^{\rm c}$.

Hence, in view of Theorem \ref{key real}(i), one knows that $f$ is cyclotomic if and only if the restriction $\iota_p(f^{1+\tau})$ of $f^{1+\tau}$  belongs to $(1+\tau)\mathcal{F}_{(p)}^{\rm c}$.

Now the definition (given in \S\ref{proof section}) of the map $\kappa$  in Theorem \ref{main result} implies that the projection $\kappa(f)_p$ of $\kappa(f)$ to $\widehat{R}^p(1+\tau)$ is the unique element with the property that for each $m$ in $\mathbb{N}(p)$ one has $f^{1+\tau}(m) = \Phi(m)^{\kappa(f)_p}$ in $V(m)_p$. It follows that $f$ is cyclotomic if and only if $\kappa(f)_p$ belongs to $R(1+\tau)$.

We next recall from Remark \ref{explicit kappa} that, for any choice of elements $a_{m,b}(f)$ as in (\ref{key eq 2}), the sequence  $(\pi^{mp^b}_m(a_{m,b}(f)))_{b\in \mathbb{N}}$ is convergent in $R_{m,p}^+$ and that $\kappa(f)_p$ is the element of the limit $\widehat{R}^p(1+\tau)\cong \varprojlim_m R_{m,p}^+$ that corresponds to $(\lim_{b\to \infty}\pi^{mp^b}_m(a_{m,b}(f)))_m$.

This implies that $\kappa(f)_p$ belongs to $R(1+\tau)$ if and only if for every $m$ in $\mathbb{N}(p)$ the limit $\lim_{b\to \infty}\pi^{mp^b}_m(a_{m,b}(f))$ belongs to $R_m^+$.

Now, if for each $b$ in $\mathbb{N}$ and $g$ in $G_m^+$, we define $p$-adic integers $c_{m,b,g}(f)$ by the equality
\[ \pi^{mp^b}_m(a_{m,b}(f)) = \sum_{g \in G_m^+} c_{m,b,g}(f)\cdot g,\]
then each sequence $(c_{m,b,g}(f))_b$ is convergent in $\ZZ_p$ and one has
\[ \lim_{b\to \infty}\pi^{mp^b}_m(a_{m,b}(f)) = \sum_{g \in G_m^+}(\lim_{b\to \infty}c_{m,b,g}(f))\cdot g.\]

In addition, for any element $g$ of $G_m^+$ and any element $\sigma$ of $\Gal(\QQ^c/\QQ)$ that restricts on $\QQ(m)^+$ to give $g^{-1}$ one can take the elements $a_{m,b}(f^\sigma)$ in (\ref{key eq 2}) to be $\sigma\cdot a_{m,b}(f)$ and in this way compute that $c_{m,b,g}(f) = c_{m,b,e}(f^\sigma)$ where $e$ is the identity element of $G_m^+$.

This observation reduces us to showing that for every $m$ in $\mathbb{N}(p)$ and every $f$ in $\mathcal{F}^{\rm d}$ the limit $c:= \lim_{b\to \infty}c_{m,b,e}(f)$ belongs to $\ZZ$ if and only if for any sufficiently large integer $k$ at least one of the sets $\{ |a_{m,n}(f)|_k\}_{n> k}$ and $\{ |-\!a_{m,n}(f)|_k\}_{n > k}$ is bounded.

If, firstly, $c$ is a non-negative integer, then, since $c$ is the coefficient of $e$ in the limit $\lim_{n\to\infty}\pi^{mp^n}_m(a_{m,n}(f))$ one must have $|a_{m,n}(f)|_k = c$ for all sufficiently large $n$ and so the set $\{|a_{m,n}(f)|_k\}_{n>k}$ is indeed bounded.

Similarly, if $c$ is a negative integer, then $|-a_{m,n}(f)|_k$ must be equal to $-c$ for all sufficiently large $n$ and so $\{|-\!a_{m,n}(f)|_k\}_{n>k}$ is bounded.

To prove the converse we fix $k$ such that either $\{ |a_{m,n}(f)|_k\}_{n>k}$ or $\{ |-\!a_{m,n}(f)|_k\}_{n>k}$ is bounded and choose a strictly increasing sequence of natural numbers $(n_i)_i$ with $n_i > k$ and $c_{m,n_i,e}(f) \equiv c$ (mod $p^{n_i-k}$) for each $i$.

Then, with $c(i)$ denoting the unique integer in $\{a:0 \le a < p^{n_i-k}\}$ with $c(i) \equiv c_{m,n_i,e}(f)$ (mod $p^{n_i-k}$), one finds that the element $(c(i))_i$ belongs to $\ZZ_p = \varprojlim_i \ZZ/(p^{n_i-k})$ and is equal to $c$.

Now, if  $\{ |a_{m,n}(f)|_k\}_{n>k}$ is bounded, then the definition of $|a_{m,n}(f)|_k$ ensures that the increasing sequence $(c(i))_i$ is eventually constant. This implies that $c$ is equal to $c(i)$ for any large enough value of $i$ and so is a non-negative integer.

Similarly, if $\{ |-\!a_{m,n}(f)|\}_{n >k}$ is bounded, then we can repeat the above argument with each term $a_{m,n}(f)$ replaced by $-\!a_{m,n}(f)$ to deduce that $-c$ is a non-negative integer, as required to complete the proof.
\end{proof}

\begin{remark}\label{p-b indep}{\em It is easily seen that the condition of being $p$-bounded is independent of the precise choice of elements $a_{m,n}(f)$ that occur in the equality (\ref{key eq 2}). To be specific, if one fixes $m$ and considers any other choice $\{a'_{m,n}(f)\}_n$ of such elements, then for each $n$ one has $a_{m,n}(f) - a'_{m,n}(f) \in I_{mp^n,p}$. In particular, if one chooses $k$ to be greater than the integer $b_0$ that occurs in the proof of Lemma \ref{derived proj limit lemma}, then the latter argument shows that for every $n > k$ one has
 $\pi^{mp^n}_m(a_{m,n}(f) - a'_{m,n}(f)) \in p^{n-k}\cdot R_{m,p}^+$ and hence that $|a_{m,n}(f)|_k = |a'_{m,n}(f)|_k$.}\end{remark}

\begin{remark}{\em For the norm compatible family discussed in Lemma \ref{nc-nd exam}(i) it can be shown that there is no natural number $k$ for which either of the sets $\{ |\Pi_n|_k\}_{n>k}$ or $\{ |-\!\Pi_n|_k\}_{n>k}$ is bounded, where each $\Pi_n$ is regarded as an element of $R_{qp^n}^+$. In view of Theorem \ref{new criterion}, Coleman's Conjecture therefore predicts that no such family can arise as the restriction of a distribution. However, aside from the special case considered in Lemma \ref{nc-nd exam}(ii), we have not yet been able to verify this prediction. }\end{remark}

\subsection{}In this section we give a concrete description of the image of the map $\kappa$ in Theorem \ref{main result} and thereby derive an alternative criterion for the validity of Coleman's Conjecture.

We shall also then give an interpretation of this criterion in terms of the Galois structure of Selmer groups of $\mathbb{G}_m$.

For each $m$ in $\mathbb{N}^*$ we consider the subgroup
\[ V'(m):= (E(m)')^{1+\tau}\]
of the (torsion-free) group $V(m)$.


\subsubsection{}We set $\Gamma_m := G_m^+$ and recall the idempotent $e_m$ of $\QQ[\Gamma_m]$ defined in (\ref{n idem}). Then Lemma \ref{useful 3} implies that
$\varepsilon_m = e_m\cdot\varepsilon_m$ in $\QQ\otimes_\ZZ V'(m)$ and so one obtains a $\ZZ[\Gamma_m]$-submodule of $\QQ[\Gamma_m]e_m$ by setting
\[ J_m :=  \{x \in \QQ[\Gamma_m]e_m: \varepsilon_m^x\in V'(m)\}.\]

We also write $\widehat{\ZZ}$ for the profinite completion of $\ZZ$ and define a submodule of $\widehat{\ZZ}[\Gamma_m]$ by setting
\[ X_m := \prod_{\ell \mid m} I_{m,\ell}\cdot \prod_{\ell\nmid m}\ZZ_\ell[\Gamma_m].\]

In the next result (and its proof) we identify $\ZZ[G_m](1+\tau)$ with $\ZZ[\Gamma_m]$ via the map sending each $x(1+\tau)$ to the image of $x$ under the projection map $\ZZ[G_m] \to \ZZ[\Gamma_m]$. In this way we regard $J_m$ and $X_m$ as subsets of $\QQ[G_m](1+\tau)$ and $\widehat{\ZZ}[G_m](1+\tau)$ respectively.

\begin{theorem}\label{image result}\
\begin{itemize}
\item[(i)] In $\QQ\otimes_\ZZ\widehat R(1+\tau)$ one has
\[2\bigl(\widehat R(1+\tau) \cap \prod_{m\in \mathbb{N}^\ast}( J_m + X_m)\bigr) \subseteq \im(\kappa) \subseteq \widehat R(1+\tau) \cap \prod_{m\in \mathbb{N}^\ast}( J_m + X_m).\]
\item[(ii)] If Coleman's Conjecture is valid, then
\[ 2\bigl(\widehat R(1+\tau) \cap \prod_{m\in \mathbb{N}^\ast}( J_m + X_m)\bigr) \subseteq R(1+\tau).\]
\item[(iii)] After inverting $2$, Coleman's Conjecture is valid if and only if 
\[\ZZ\bigl[1/2\bigr]\otimes_\ZZ R(1+\tau) = \ZZ\bigl[1/2\bigr]\otimes_\ZZ\bigl(\widehat R(1+\tau) \cap \prod_{m\in \mathbb{N}^\ast}( J_m + X_m)\bigr).\]
\end{itemize}\end{theorem}

\begin{proof} Claims (ii) and (iii) both follow directly from claim (i) and the commutative diagram given in Theorem \ref{main result}.

To prove claim (i) we set $\ZZ_{\langle m\rangle} := \prod_{\ell \mid m}\ZZ_\ell$ for each natural number $m$. We also write $\Delta_m$ for the diagonal map $V'(m) \to \ZZ_{\langle m\rangle}\otimes_\ZZ V'(m)$ (which is injective since $V'(m)$ is torsion-free) and $\pi_m$ for the natural
projection $\widehat{R}(1+\tau) \to \ZZ_{\langle m\rangle}[G_m](1+\tau)= \ZZ_{\langle m\rangle}[\Gamma_m]$.

Then for each $f$ in $\mathcal{F}^{\rm d}$ and each $m$ in $\mathbb{N}^\ast$ the element $f(m)^{1+\tau}$ belongs to $V'(m)$ and so Remark \ref{more general ann} implies the existence of an element $j_m = j_{f,m}$ in $J_m$ with $f(m)^{1+\tau} = \varepsilon_m^{j_m}$ in $\QQ\otimes_\ZZ V'(m)$.

Thus, since the definition of $\kappa$ implies that in $\ZZ_{\langle m\rangle}\otimes_\ZZ V'(m)$ one has
\[ \Delta_m(f(m)^{1+\tau}) = \varepsilon_m^{\pi_m(\kappa(f))}, \]
the difference $\pi_m(\kappa(f))-j_m$ must belong to the submodule $\prod_{\ell \mid m}I_{m,\ell}$ of $\ZZ_{\langle m\rangle}[\Gamma_m]$.

This implies that the projection of $\kappa(f)$ to $\widehat{\ZZ}[G_m](1+\tau) \subset \QQ\otimes_\ZZ \widehat{\ZZ}[G_m](1+\tau)$ belongs to the submodule $J_m + X_m$. 
%
%
 Since this is true for all $f$ and all $m$ it follows that $\im(\kappa)$ is contained in $\widehat R(1+\tau) \cap \prod_{m}( J_m + X_m)$, as claimed.

To complete the proof of claim (i) we must show that for any element $y = (y_m)_m$ in $\widehat R(1+\tau) \cap \prod_{m}( J_m + X_m)$ one has $2\cdot y\in \im(\kappa)$.

Now for each $m$ in $\mathbb{N}^\ast$ one has $y_m = j_m + x_m$ with $j_m\in J_m$ and $x_m \in X_m$ and hence also
\[ \varepsilon_m^{\pi_m(y)} = \varepsilon_m^{j_m}\in \im(\Delta_m) \subseteq \ZZ_{\langle m\rangle}\otimes_\ZZ V'(m).\]

We then write $f_m$ for the unique element of $V'(m)$ with $\Delta_m(f_m) = \varepsilon_m^{\pi_m(y)}$ and claim that the unique $\Gal(\QQ^c/\Q)$-equivariant map $f: \mu_\infty^*\to \QQ^{c,\times}$ that sends $\zeta_m$ to $f_m$ belongs to $\mathcal{F}^{\rm d}$ and is such that $\kappa(f) = 2\cdot y$.

To prove $f$ belongs to $\mathcal{F}^{\rm d}$ we must show that it satisfies the distribution relations (\ref{basic facts}). To do this we fix $m$ in $\mathbb{N}^*$ and a prime $\ell$ and note there is a commutative diagram
\[
\xymatrix{ V'(m\ell)\ar[d] \ar@{^{(}->}[r] & \ZZ_{\langle m\rangle}\otimes_\ZZ V'(m\ell) \ar[d]\\
V'(m) \ar@{^{(}->}[r]^{\!\!\!\!\!\!\!\!\!\!\!\Delta_m} & \ZZ_{\langle m\rangle}\otimes_\ZZ V'(m)}
\]
in which the upper horizontal arrow denotes the natural diagonal map and the vertical arrows are induced by ${\rm N}^{m\ell}_m$.

By using this diagram one computes in $\ZZ_{\langle m\rangle}\otimes_\ZZ V'(m)$ that
\[ \Delta_m({\rm N}^{m\ell}_m(f(m\ell))) = {\rm N}^{m\ell}_m(\varepsilon_{m\ell}^{y_{m\ell}}) = {\rm N}^{m\ell}_m(\varepsilon_{m\ell})^{y_{m}}
 = \begin{cases} \Delta_m(f(m)), &\text{ if $\ell \mid m$,}\\
\Delta_m(f(m)^{1-\sigma_\ell}), &\text{ if $\ell\nmid m$.}\end{cases}
  \]
Here the second equality is valid because $y$ belongs to $\widehat R(1+\tau)$ and the third follows by combining the standard distribution relations for $\varepsilon_{m\ell}$ with the fact that in $\ZZ_{\langle m\rangle}\otimes_\ZZ V'(m)$ one has
$\Delta_m(f(m)) = \Delta_m(f_m) = \varepsilon_{m}^{\pi_{m}(y)}$.

This shows $f$ belongs to $\mathcal{F}^{\rm d}$. To show $\kappa(f) = 2\cdot y$ it is then sufficient to note that for every prime $p$ and every  $m$ in $\mathbb{N}(p)$ one has  $f(m)^{1+\tau} = f_m^{1+\tau} = f_m^2 = \varepsilon_{m}^{2\cdot y}$ in $V(m)_p$.
\end{proof}

\begin{remark}\label{2 difficult}{\em The difficulty with the factor of $2$ in Theorem \ref{image result}(i) arises because knowing that a distribution $f$ is valued in the groups $(E(m)')^{1+\tau}$ need not, a priori, imply that it belongs to $(1+\tau)\mathcal{F}^{\rm d}$ (and see also Remark \ref{tate coh rem} in this regard). }\end{remark}

\subsubsection{}In this section we explain the connection between the result of Theorem \ref{image result} and the Galois structure of Selmer groups.


To do this we recall (from, for example, \cite[\S5.2.1]{bst}) that for any number field $F$ and any finite set of places $S$ of $F$ that contains all archimedean places, the $S$-relative Selmer group ${\rm Sel}_{S}(F)$ for $\GG_m$ over $F$ is the cokernel of the homomorphism
\[{\prod}_{w}\ZZ \longrightarrow \Hom_\ZZ(F^{\times},\ZZ), \]
where $w$ runs over all places of $F$ that do not belong to $S$ and the arrow sends $(x_w)_w$ to the map $(u \mapsto \sum_{w}{\rm ord}_w(u)x_w)$ with ${\rm ord}_w$ the normalised additive valuation at $w$.

We further recall that this group is an analogue for $\mathbb{G}_m$ of the integral Selmer groups of abelian varieties defined by Mazur and Tate in \cite{mt}, that it lies in a canonical exact sequence
\begin{equation}\label{selmer ses} 0 \to \Hom_\ZZ({\rm Pic}(\mathcal{O}_{F,S}),\QQ/\ZZ) \to {\rm Sel}_{S}(F) \to \Hom_\ZZ(\mathcal{O}^{\times}_{F,S},\ZZ) \to 0\end{equation}
and that it has a subquotient that is canonically isomorphic to the Pontryagin dual of the ideal class group of $F$.

For each natural number $m$ we now set
\[ {\rm Sel}(m) := {\rm Sel}_{S_m}(\QQ(m)^+),\]
where $S_m$ denotes the set of places of $\QQ(m)^+$ comprising all archimedean places and all places that divide $m$.

For each finitely generated $\Gamma_m$-module $M$ we write ${\rm Fit}_{\Gamma_m}^1(M)$ for the first Fitting ideal of $M$, as defined by Northcott in \cite{north}.

Then the exact sequence (\ref{selmer ses}) combines with the equality of \cite[(45)]{bst} and the argument of Lemma \ref{useful 3} to imply that every element of ${\rm Fit}^1_{\Gamma_m}({\rm Sel}(m))$ is stable under multiplication by the idempotent $e_m$.

It is therefore natural to consider the `$\ZZ[\Gamma_m]$-conductor' of ${\rm Fit}_{\Gamma_m}^1({\rm Sel}(m))$ in $\QQ[\Gamma_m]e_m$ that is obtained by  setting
\[ K_m := \{x \in \QQ[\Gamma_m]e_m: x\cdot {\rm Fit}_{\Gamma_m}^1({\rm Sel}(m)) \subseteq \ZZ[\Gamma_m]\}. \]

This set is a full $\ZZ[\Gamma_m]$-sublattice of $\QQ[\Gamma_m]e_m$ and the following result shows that, at least after inverting $2$, it agrees with the module $J_m$ that occurs in the explicit criterion for the validity of Coleman's Conjecture given in Theorem \ref{image result}.

\begin{proposition}\label{image result 2} For each $m$ in $\mathbb{N}^*$ one has $\ZZ\bigl[1/2\bigr]\otimes_\ZZ J_m = \ZZ\bigl[1/2\bigr]\otimes_\ZZ K_m$.\end{proposition}

\begin{proof} We set $A' := \ZZ\bigl[1/2\bigr]\otimes_\ZZ A$ for any abelian group $A$ (so that $\ZZ'$ denotes $\ZZ\bigl[1/2\bigr]$).

For each torsion-free $\ZZ'[\Gamma_m]$-module $X$ we write $X^{e_m}$ for the submodule comprising all elements $x$ that satisfy $x = e_m\cdot x$ in $\QQ\otimes_\ZZ X$. We also set $U_m := (\mathcal{O}_{\QQ(m)^+,S_m}^\times)'$.

Then the key fact we use is that Sano, Tsoi and the first author have shown in \cite[Th. A]{bst} that there is a canonical isomorphism of finite $\ZZ'[\Gamma_m]$-modules of the form
\begin{equation}\label{bst iso} U_m^{e_m}/\langle \varepsilon_m\rangle \cong \Hom_\ZZ\bigl( \ZZ'[\Gamma_m]^{e_m}/Y_m,\QQ/\ZZ\bigr),\end{equation}
%
%
where we write $\langle \varepsilon_m\rangle$ for $\ZZ'[\Gamma_m]\cdot \varepsilon_m$ and set $Y_m := {\rm Fit}_{\Gamma_m}^1({\rm Sel}(m))'$ (but see Remark \ref{no T} below).

To compute the right hand side of (\ref{bst iso}) we note that, since $\ZZ'[\Gamma_m]^{e_m}$ is $\ZZ'$-free and $\ZZ'[\Gamma_m]^{e_m}/Y_m$ is finite, taking $\ZZ'$-linear duals of the tautological exact sequence
\[ 0 \to Y_m \to \ZZ'[\Gamma_m]^{e_m} \to \ZZ'[\Gamma_m]^{e_m}/Y_m\to 0\]
gives a canonical exact sequence
\[ 0 \to \Hom_{\ZZ'}(\ZZ'[\Gamma_m]^{e_m},\ZZ') \to \Hom_{\ZZ'}(Y_m,\ZZ') \to \Hom_{\ZZ}(\ZZ'[\Gamma_m]^{e_m}/Y_m,\QQ/\ZZ)\to 0.\]

In addition, the argument of \cite[Prop. 3.29(ii)]{bst} shows that the canonical identifications
\[ \Hom_{\QQ}(\QQ[\Gamma_m]^{e_m},\QQ)\cong \Hom_{\QQ[\Gamma_m]}(\QQ[\Gamma_m]^{e_m},\QQ[\Gamma_m]) \cong \QQ[\Gamma_m]e_m\]
send the sublattices $\Hom_{\ZZ'}(\ZZ'[\Gamma_m]^{e_m},\ZZ')$ and $\Hom_{\ZZ'}(Y_m,\ZZ')$ to $\ZZ'[\Gamma_m]e_m$ and $K_m'$ respectively and so the above exact sequence gives a canonical isomorphism
\[ \theta:  K_m'/(\ZZ'[\Gamma_m]e_m)\cong \Hom_{\ZZ}(\ZZ'[\Gamma_m]^{e_m}/Y_m,\QQ/\ZZ).\]

The description given in \cite[Rem. 3.31]{bst} then implies that the isomorphism
\[ K_m'/(\ZZ'[\Gamma_m]e_m) \cong U_m^{e_m}/\langle \varepsilon_m\rangle\]
obtained by composing $\theta$ with the inverse of (\ref{bst iso}) is induced by the isomorphism $\QQ[\Gamma_m]e_m\cong \QQ[\Gamma_m]\varepsilon_m$ that sends each element $x$ to $\varepsilon_m^x$.

This fact implies $K_m'$ is equal to $\{y \in \QQ[\Gamma_m]e_m: \varepsilon_m^y \in U_m^{e_m}\}$ and the claimed result now follows since a straightforward exercise shows that the latter set is equal to $J_m'$.
\end{proof}

\begin{remark}\label{no T} {\em We take the isomorphism (\ref{bst iso}) to be that induced in the canonical way by the non-degenerate pairing constructed in \cite[Th. A]{bst}. In addition, we note that the modules in the latter result differ slightly from those used above in that an auxiliary set of places of $\QQ(m)^+$ is fixed and used to specify a torsion-free subgroup of $\mathcal{O}_{\QQ(m)^+,S_m}^\times$. However, the use of such a set is unnecessary after inverting $2$ since the group $U_m = (\mathcal{O}_{\QQ(m)^+,S_m}^\times)'$ is torsion-free and, using this fact, the same arguments as in loc. cit. lead to a non-degenerate pairing involving the modules $U_m$ and ${\rm Fit}_{\Gamma_m}^1({\rm Sel}(m))'$, and hence to the isomorphism (\ref{bst iso}).}\end{remark}

\begin{remark}{\em We have inverted $2$ in the result of Proposition \ref{image result 2} solely to avoid technical difficulties of the form discussed in Remarks \ref{2 difficult} and \ref{no T}. However, we believe that, with more effort, it should in principle be possible to deal with such difficulties in the context of Proposition \ref{image result 2}. }\end{remark}

\section{Divisible distributions}\label{divisible section}

We say that a distribution is `$p$-divisible' for some prime $p$, respectively is `divisible', if its image in $\mathcal{F}_{\rm tf}^{\rm d}/\mathcal{F}^{\rm c}$ is $p$-divisible, respectively is divisible, and we write $\mathcal{F}^{\rm d}_{\rm pdiv}$ and $\mathcal{F}^{\rm d}_{\rm div}$ for the $R$-submodules of $\mathcal{F}^{\rm d}$ comprising all such distributions.

Irrespective of the validity, or otherwise, of Coleman's Conjecture, such distributions are closely related to circular distributions and share many of the same properties (see Proposition \ref{explicit cors} below).

In Theorem \ref{key real}(iii) we described an explicit criterion for {\em every} distribution to be $p$-divisible.

In this section we explain how the map $\kappa$ in Theorem \ref{main result} also leads to an explicit criterion for a given distribution to be either $p$-divisible or divisible.

\subsection{}For each $f$ in $\mathcal{F}^{\rm d}$ and each prime $p$ we write $\kappa(f)_p$ for the image of $\kappa(f)$ under the natural projection $\widehat{R}(1+\tau) \to \widehat{R}^{p}(1+\tau)$.

\begin{proposition}\label{p div cond} For each distribution $f$ in $\mathcal{F}^{\rm d}$ the following claims are valid.

\begin{itemize}
\item[(i)] $f$ is $p$-divisible if and only if for each $m$ in $\mathbb{N}^*$ one has
\[ f(m)^{1+\tau} = \Phi(m)^{\kappa(f)_p}\]
 in $V(m)_p$.
\item[(ii)]  $f$ is divisible if and only if for each $m$ in $\mathbb{N}^*$ one has
\[ f(m)^{1+\tau} = \Phi(m)^{\kappa(f)}\]
in $\widehat{\ZZ}\otimes_\ZZ V(m)$.
 \end{itemize}
 \end{proposition}

\begin{proof} To prove claim (i) we assume first that $f$ is $p$-divisible. Then for each natural number $n$ there exists an element $s(n)'$ of $R$ and functions $\delta_n$ in $\mathcal{D} = \mathcal{F}^{\rm d}_{\rm tor}$ and $f_n$ in $\mathcal{F}^{\rm d}$ such that
\[ f\cdot\Phi^{-s(n)'} = \delta_n\cdot f_n^{p^n}.\]

Setting $s(n) := (1+\tau)s(n)'$, this fact implies that $\kappa(f) - s(n) = \kappa(f\cdot\Phi^{-s(n)'}) = p^n\cdot \kappa(f_n)$. For each $m$ in $\mathbb{N}^*$ this in turn gives equalities in $V(m)_p$ of the form
\begin{align*}
f(m)^{1+\tau} = &\Phi(m)^{s(n)_p}\cdot f_n(m)^{p^n(1+\tau)}\\ = &\Phi(m)^{\kappa(f)_p}\cdot \Phi(m)^{-\kappa(f)_p + s(n)_p}\cdot  f_n(m)^{p^n(1+\tau)}\\ = &\Phi(m)^{\kappa(f)_p}\cdot (\Phi(m)^{-\kappa(f_n)_p}\cdot f_n(m)^{(1+\tau)})^{p^n}
\end{align*}

It follows that $f(m)^{1+\tau}\cdot \Phi(m)^{-\kappa(f)_p}$ belongs to $V(m)_p^{p^n}$ for all $n$ and hence that $f(m)^{1+\tau} =  \Phi(m)^{\kappa(f)_p}$, as claimed.

To prove the converse we assume that $f(m)^{1+\tau} = \Phi(m)^{\kappa(f)_p}$ in $V(m)_p$ for all $m$ in $\mathbb{N}^*$. For each $n$ in $\mathbb{N}$ we can then fix an $r(n)$ in $R(1+\tau)$ such that $\kappa(f) - r(n) = p^n\cdot s(n)$ for some $s(n)$ in $\widehat{R}(1+\tau)$.

For each $m$ in $\mathbb{N}^*$ one therefore has
 \[ f(m)^{1+\tau} = \Phi(m)^{\kappa(f)_p} = \Phi(m)^{r(n)}\cdot \Phi(m)^{p^n\cdot s(n)_p}\]
in $V(m)_p$ and hence $(f^{1+\tau}\Phi^{-r(n)})(m) \in V(m) \cap (V(m)_p)^{p^n}$. Since the group $V(m)_p/V(m)$ is uniquely $p$-divisible, it follows that
  $(f^{1+\tau}\Phi^{-r(n)})(m)$ belongs to $V(m)^{p^n}$ for all $m$.

In particular, since each group $V(m)$ is torsion-free, there exists a unique element $h_{n,m}$ of $V(m)$ with $(f^{1+\tau}\Phi^{-r(n)})(m) = (h_{n,m})^{p^n}$ and the unique map $h_n$ in $\mathcal{F}$ with $h_n(m) := h_{n,m}$ for all $m$ in $\mathbb{N}^*$ inherits the property of being a distribution from $f^{1+\tau}\Phi^{-r(n)}$.

But then the equality of functions
\begin{equation}\label{lucky} (f^{1+\tau}\Phi^{-r(n)})^2 = (f^{1+\tau}\Phi^{-r(n)})^{1+\tau} = (h_n^{1+\tau})^{p^n}\end{equation}
implies that the image of $f^{2(1+\tau)}$ is $p$-divisible in $(1+\tau)\mathcal{F}^{\rm d}_{\rm tf}/(1+\tau)\mathcal{F}^{\rm c}$ and hence, via the isomorphism in Theorem \ref{reduction result}(iii), that $f^2$ is $p$-divisible. If $p$ is odd, then this directly implies that $f$ is $p$-divisible.

On the other hand, if $p=2$, then since all the functions in (\ref{lucky}) are valued in torsion-free groups, this equality implies
 that $f^{1+\tau}\Phi^{-r(n)} = (h_n^{1+\tau})^{2^{n-1}}$ and, by the same argument as above, this implies that $f$ is $2$-divisible.

Claim (ii) follows directly from claim (i) and the obvious fact that $f$ is divisible if and only if it is $p$-divisible for all primes $p$. \end{proof}

\subsection{}\label{final sec}In this final section we verify that divisible distributions have many of the same properties as do cyclotomic distributions, thus justifying an observation made in the Introduction.

We write $\mathcal{F}^{\rm sd}_{{\rm div}}$ for the submodule of $\mathcal{F}^{\rm sd}$ comprising all strict distributions whose image in $\mathcal{F}^{\rm d}_{\rm tf}/\mathcal{F}^{\rm c}$ is divisible.

We shall also say that two subsets $\mathcal{F}_1$ and $\mathcal{F}_2$ of $\mathcal{F}$ `have the same values' if for each $n$ in $\mathbb{N}^\ast$ the sets $\{f(n): f \in \mathcal{F}_1\}$ and $\{f(n): f \in \mathcal{F}_2\}$ generate the same $R$-module.

\begin{proposition}\label{explicit cors}\
\begin{itemize}
\item[(i)] One has $\mathcal{F}^{\rm d}_{{\rm div}} = \mathcal{D} + \mathcal{F}^{\rm sd}_{{\rm div}}$.
\item[(ii)] The groups $\mathcal{F}^{\rm c}$, $\mathcal{F}_{\rm div}^{\rm sd}$ and $\mathcal{F}_{\rm div}^{\rm d}$ have the same values.

\item[(iii)] If $f$ is any element of $\mathcal{F}_{\rm div}^{\rm d}$, then there exist precisely two finite products $\delta$ of functions of the form $\delta_\Pi$, for suitable sets of odd primes $\Pi$,  for which, at each $n$ in $\mathbb{N}^\ast$, each prime $\ell$ that does not divide $n$, each $\varepsilon$ in $\mu_\ell$ and each $\zeta$ in $\mu_n^*$ one has
\begin{equation}\label{defining prop22} \delta(\varepsilon\cdot \zeta)\cdot f(\varepsilon\cdot \zeta) \equiv \delta(\zeta)\cdot f(\zeta) \,\,\text{ modulo all primes above $\ell$.}\end{equation}
\item[(iv)] If $f$ is any element of $\mathcal{F}^{\rm d}_{\rm div}$, then there exists an integer $v$ that depends only on $f$ and is such that for all primes $p$ and all natural numbers $n$, the valuation of $f(p^n)$ at the unique prime of $\QQ(p^n)$ above $p$ is equal to $v$.
\item[(v)] If $f$ is any element of $\mathcal{F}^{\rm d}_{\rm div}$, then for any odd prime $p$ there exists an element of $\mathcal{F}^{\rm c}$ that (may depend on $p$ and) agrees with $f$ when evaluated at any root of unity of $p$-power order.
\end{itemize}
\end{proposition}

\begin{proof} The key point is to note is that if $f$ belongs to $\mathcal{F}_{\rm div}^{\rm d}$, then for each natural number $t$ there exist maps $\delta_{f,t}$ in $\mathcal{F}^{\rm d}_{\rm tor}=\mathcal{D}$ (cf. Remark \ref{more concept}), $f_t$ in $\mathcal{F}^{\rm d}$ and $f^{\rm c}_t$ in $\mathcal{F}^{\rm c}$ such that
\begin{equation}\label{key equality} f = \delta_{f,t}\cdot f^{\rm c}_t \cdot (f_t)^t.\end{equation}

To prove claim (i) we note that $f$ belongs to $\mathcal{F}^{\rm sd}$ if and only if for all natural numbers $n$, all primes $\ell$ that are coprime to $n$, all $\varepsilon$ in $\mu_\ell$ and all $\zeta$ in $\mu_n^*$ it satisfies the congruence (\ref{defining prop2}).

To check this we write $t_{n,\ell}$ for any choice of $n$ and $\ell$ as above for the product of the orders of the multiplicative groups of the residue fields of $\QQ(n)$ at each prime above $\ell$.

Then, for any given value of $n$ and $\ell$, the congruence (\ref{defining prop2}) is clearly satisfied by the $t_{n,\ell}$-th power of any map in $\mathcal{F}^{\rm d}$. We also know that this congruence is satisfied by any map in $\mathcal{F}^{\rm c}$ and that the square of any map in $\mathcal{D}$ is trivial.

In particular, if for any given $n$ and $\ell$ we combine these facts with the equality (\ref{key equality}) with $t = t_{n,\ell}$, we find that $f^2$ satisfies the congruence (\ref{defining prop2}) for this choice of $n$ and $\ell$.

Thus, since $n$ and $\ell$ are arbitrary, we deduce that the square of any map $f$ in $\mathcal{F}_{\rm div}^{\rm d}$ belongs to $\mathcal{F}^{\rm sd}$, and hence also to $\mathcal{F}_{\rm div}^{\rm sd}$.

This shows that the quotient of $\mathcal{F}_{\rm div}^{\rm d}$ by $X := \mathcal{D} + \mathcal{F}_{\rm div}^{\rm sd}$ has exponent dividing $2$. On the other hand, since $\mathcal{F}^{\rm c}\subseteq X$, the group $\mathcal{F}_{\rm div}^{\rm d}/X$ is isomorphic to a quotient of the image of
$\mathcal{F}_{\rm div}^{\rm d}$ in $\mathcal{F}_{\rm tf}^{\rm d}/\mathcal{F}^{\rm c}$ and so is divisible. It follows that $\mathcal{F}_{\rm div}^{\rm d}/X$ is trivial, and hence that claim (i) is valid.

Noting $\mathcal{F}^{\rm c} \subseteq \mathcal{F}_{\rm div}^{\rm sd} \subseteq \mathcal{F}_{\rm div}^{\rm d}$, to prove claim (ii) it suffices to show that for any $f$ in $\mathcal{F}_{\rm div}^{\rm d}$ and any $n$ in $\mathbb{N}^*$ the element $f(n)$ of $E(n)'$ belongs to the $R_n$-submodule generated by $1-\zeta_n$.

By the same argument as in \S\ref{semi-local}, it is thus enough to show that the image $f(n)'$ of $f(n)$ in $E(n)'_{\rm tf}$ belongs to the group $C(n)_{\rm tf}'$, where the group of cyclotomic numbers $C'(n)$ is as defined in \S\ref{cyclo section}.

But for any $t$ in $\mathbb{N}$ the decomposition (\ref{key equality}) implies that $f(n)'$ belongs to
$C'(n)_{\rm tf}\cdot (E(n)'_{\rm tf})^t$. Since $t$ is arbitrary, this in turn implies that $f(n)'$ belongs to
$C'(n)_{\rm tf}$, as required to prove claim (ii).

Claim (i) implies that for any $f$ in $\mathcal{F}_{\rm div}^{\rm d}$ there exists a map $\delta$ in $\mathcal{D}$ such that $\delta\cdot f$ belongs to $\mathcal{F}^{\rm sd}$ or, equivalently, such that the congruences (\ref{defining prop2}) are satisfied.

In addition, if $\delta'$ is any other map in $\mathcal{D}$ with this property, then $\delta^{-1}\delta' = (\delta\cdot f)^{-1}(\delta'\cdot f)$ belongs to $\mathcal{D}\cap \mathcal{F}^{\rm sd}$ and so is either trivial or equal to the map $\delta_{\rm odd}$ described in the Introduction. This proves claim (iii).

For each prime $p$ and each natural number $m$ we write ${\rm val}_{p^m}$ for the standard valuation of $\QQ(p^m)$ at the unique prime ideal above $p$.

 We further note that, for any $r$ in $R$, the valuation ${\rm val}_{p^m}((1-\zeta_{p^m})^r)$ is equal to the image of $r$ under the natural projection map $R \to \ZZ$ and so is independent of both $p$ and $m$. It is also clear that for any $\delta$ in $\mathcal{D}$ one has ${\rm val}_{p^m}(\delta(p^m)) = 0$.

For any $f$ in $\mathcal{F}_{\rm div}^{\rm d}$, these observations combine with (\ref{key equality}) to imply, for any given primes $p$ and $q$, and any given natural numbers $t$, $m$ and $n$, that there are congruences modulo $t$ of the form
\[{\rm val}_{p^{m}}(f(p^{m})) \equiv {\rm val}_{p^{m}}(f_t^{\rm c}(p^{m})) = {\rm val}_{q^{n}}(f_t^{\rm c}(q^{n})) \equiv {\rm val}_{q^{n}}(f(q^{n})).\]
Since $t$ is arbitrary, this implies ${\rm val}_{p^{m}}(f(p^{m})) = {\rm val}_{q^{n}}(f(q^{n}))$ for all $p$, $q$, $m$ and $n$ and so proves claim (iv).

To prove claim (v) we fix an odd prime $p$ and write $\mu^*_{p^\infty}$ for the set of non-trivial roots of unity of $p$-power order.

We recall that the result of \cite[Th. B]{Seo4} implies the existence of a natural number $m$ that depends only on $p$ and is such that for any $f$ in $\mathcal{F}^{\rm d}$ there exists a map $g_{f,m}$ in $\mathcal{F}^{\rm c}$ which agrees with $f^m$ on $\mu^*_{p^\infty}$.

Next we note that if $\delta$ is any element of $\mathcal{D}$, then there exists an integer $a$ (depending on $\delta$) such that $\delta(p^m) = (-1)^a$ for every $m$.

This shows that $\delta$ agrees on $\mu^*_{p^\infty}$ with the map $\Phi^{r(\delta)}$ with $r(\delta) := a(1-\tau)(p+1 - \sigma)$, where $\sigma$ is any choice of element in $\Gal(\QQ^{\rm ab}/\QQ)$ that raises each root of unity of $p$-power order to the power $p+1$.

In particular, if we now fix $f$ in $\mathcal{F}_{\rm div}^{\rm d}$ and apply the equality (\ref{key equality}) with $t = m$, then the above facts imply the existence of a map $\Phi^{r(\delta_{f,m})}\cdot f^{\rm c}_m \cdot g_{f_m,m}$ in $\mathcal{F}^{\rm c}$ that agrees with $f$ on $\mu^*_{p^\infty}$, as required to prove claim (v).
\end{proof}

\begin{remark}{\em By applying an observation of Rubin in \cite[\S4.8]{R} one can show that the function $\delta$ in Proposition \ref{explicit cors}(iii) must satisfy $\delta(m) = 1$ for all even $m$. Nevertheless, taking $f$ to be a Coleman distribution $\delta_\Pi$ shows that  one cannot always take the function $\delta$ in Proposition \ref{explicit cors}(iii) to be trivial.}\end{remark}


\end{document}